\RequirePackage{fix-cm}
\documentclass[smallcondensed]{svjour3}     
\smartqed  
\usepackage{enumitem}
\usepackage{graphicx}
\usepackage{algorithm}
\usepackage{algorithmic}
\usepackage{multirow}
\usepackage{xcolor}
\usepackage{caption}
\usepackage{amsmath}   
\DeclareMathOperator*{\argmin}{argmin}
\DeclareMathOperator*{\argmax}{argmax}
\newtheorem{assumption}{Assumption}
\newcounter{RomanNumber}
\newcommand{\MyRoman}[1]{\setcounter{RomanNumber}{#1}\Roman{RomanNumber}}
\begin{document}

\title{Dual-density-based reweighted $\ell_{1}$-algorithms for a class of $\ell_{0}$-minimization problems
\thanks{The work was supported by the National Natural Science Foundation of China under the grants 12071307 and  11771003.}
}
\author{Jia-Liang Xu         \and
        Yun-Bin Zhao 
}


\institute{Jia-Liang Xu \at
             Hua Loo-Keng Center for Mathematical Sciences, Academy of Mathematics and Systems Science, Chinese Academy of Sciences, Beijing,
Beijing 100190, China \\
              \email{xujialiang@lsec.cc.ac.cn}           
           \and
           Yun-Bin Zhao \at
              Shenzhen Research Institute of Big Data, Chinese University of Hong Kong, Shenzhen, Guangdong, China. On leave from the University of Birmingham,
Birmingham B15 2TT, United Kingdom\\
              \email{yunbinzhao@cuhk.edu.cn; y.zhao.2@bham.ac.uk}
}

\date{Received: 02 October 2019 / Accepted: date}

\maketitle

\begin{abstract}
The optimization problem with sparsity arises in many areas of science and engineering such as compressed sensing, image processing, statistical learning and data sparse approximation. In this paper, we study the dual-density-based reweighted $\ell_{1}$-algorithms for a  class of $\ell_{0}$-minimization models which can be used to model a wide range of practical problems. This class of algorithms is  based on certain convex relaxations of the reformulation of the underlying $\ell_{0}$-minimization model. Such a reformulation is a special bilevel optimization problem which, in theory,  is equivalent to the underlying $\ell_{0}$-minimization problem under the assumption of strict complementarity. Some basic properties of these algorithms are discussed, and numerical experiments have been carried out   to demonstrate the efficiency  of the proposed  algorithms. Comparison of  numerical performances of the proposed methods and the classic reweighted $\ell_1$-algorithms has also been made in this paper.

\end{abstract}

\section{Introduction}
Let $\left \| x \right \|_{0}$ denote the number of nonzero components of the vector $x. $   We consider the    $\ell_{0}$-minimization problem
\begin{equation}\label{Pnew}
\begin{array}{lcl}
  &\min\limits_{x\in R^{n}}&\left \| x \right \|_{0}\\
& $s.t.$ & \left \| y-Ax \right \|_{2}\leq \epsilon,~ Bx\leq b,
\end{array}
\end{equation}
 where $A\in R^{m\times n}$ and $B\in R^{l\times n}$ are two matrices with $m\ll n$ and  $l\leq n$, $y\in R^{m}$ and $b\in R^{l}$ are two given vectors, and $\epsilon\geq 0$ is a given parameter,  and $\left \| x \right \|_{2}=(\sum_{i=1}^{n}\left| x_{i}\right|^{2})^{1/2}$ is the $\ell_{2}$-norm of $x$.
 In compressed sensing (CS),  the parameter $\epsilon$ denotes  the level of the measurement error $\eta=y-Ax$.  Clearly,  the problem \eqref{Pnew} is to find the sparsest point in the convex set
 \begin{equation}\label{feasible set T}
T=\lbrace x:  ~  \left\| y-Ax\right\|_{2}\leq \epsilon, Bx\leq b\rbrace.
\end{equation}
The constraint $Bx\leq b$  is motivated by some practical applications. For instance, many signal recovery models  might  include  extra constraints   reflecting certain  special structures or prior information of the target signals. The model \eqref{Pnew} is general enough to  cover several important applications in compressed sensing \cite{D06,C06,CERT2006,DDEK11}, 1-bit compressed sensing \cite{gupta2010,laska2011,zhaobook2018} and statistical regression \cite{tibs2007,hoef2010,liu2010,rinaldo2009}. The following  two  models   are clearly  the special cases of \eqref{Pnew}:
 \begin{equation*}
\begin{array}{ll}
$(C1)$~\min\limits_{x} \lbrace \Vert x\Vert_{0}: ~ y=Ax\rbrace; & $(C2)$~ \min\limits_{x} \lbrace  \Vert x\Vert_{0}: ~  \left\| y-Ax\right\|_{2}\leq \epsilon\rbrace.
\end{array}
\end{equation*}
The problem (C1) is often called the standard $\ell_{0}$-minimization problem \cite{Redbook,candes2005,zhaobook2018}.
Some structured sparsity models, including the nonnegative sparsity model  \cite{candes2005,CERT2006,Redbook,zhaobook2018} and the monotonic sparsity model  (isotonic regression) \cite{greenbook},  are also  the special cases of the model \eqref{Pnew}.

Clearly, directly solving the problem \eqref{Pnew} is generally very difficult since the $\ell_{0}$-norm is a nonlinear, nonconvex  and discrete function.   Moreover, due to the analysis in \cite{xu2020}, the problem \eqref{Pnew} might have infinitely many optimal solutions so that it is needed to develop some efficient algorithms to solve the problem \eqref{Pnew}.  Some   algorithms  have been developed for some special cases of the problem  such as (C1) and (C2) over the past decade, including  convex optimization and heuristic methods \cite{DDEK11,eladbook2010,Redbook,zhaobook2018}.  For instance, by replacing the $\ell_{0}$-norm in problem \eqref{Pnew} with the $\ell_{1}$-norm,  we immediately obtain the $\ell_{1}$-minimization problem
 \begin{equation}\label{Pnew1}
\begin{array}{lll}
 &\min\limits_x &\lbrace \left \| x \right \|_{1}: ~  x\in T\rbrace.
\end{array}
\end{equation}
 A more efficient class of models than (\ref{Pnew1}) is the so-called weighted $\ell_{1}$-minimization model  \cite{reweighted1995,CWB2008,zhaoli2012,zhaobook2018}.  For  (C1) and (C2), the reweighted $\ell_1$-minimization model can be stated respectively as
\begin{equation*}
\begin{array}{ll}
$(E1)$~\min\limits_{x} \lbrace \Vert Wx\Vert_{1}: ~ y=Ax\rbrace; & $(E2)$~ \min\limits_{x} \lbrace\Vert Wx\Vert_{1}:~  \left\| y-Ax\right\|_{2}\leq \epsilon\rbrace,
\end{array}
\end{equation*}
where $W=\mathrm{diag}(w)$ is a diagonal matrix with $w\in R_{+}^{n}$ being a weight vector.  A single weighted $\ell_1$-minimization is not efficient enough to outperform the standard $\ell_1$-minimization. As a result,  the reweighted $\ell_{1}$-algorithm  has been developed, which consists of solving a series of individual weighted $\ell_{1}$-minimization  problems \cite{asif2013,asif2014,CWB2008,reweighted1995,zhaoli2012,zhaobook2018}. Taking  (C1)  as an example, this method   solves a series of the following reweighted $\ell_{1}$-problems:
$$\min\limits_{x} \lbrace  (w^{k})^{T}\vert x\vert: ~  y=Ax\rbrace,$$
where $k$ denotes the $k$th iteration and the weight $w^{k}$  is updated by a certain rule. For example, the first-order method  would yield  a good  updating scheme for $w^{k}.$ The convergence of some reweighted algorithms was shown under certain conditions \cite{cz2010,laiwang2011,zhaoli2012,zhaobook2018}.  The reweighted $\ell_{1}$-minimization may perform  better    than  $\ell_{1}$-minimization  on   sparse signal recovery when the initial point is suitably chosen (see, e.g.,  \cite{CWB2008,FSLM2009,laiwang2011,cz2010,zhaoli2012,zhaobook2018}).
Although this paper focuses on the study of reweighted algorithms, it is worth mentioning that there exist other types of  algorithms for $\ell_0$-minimization problems, which have also been widely studied in the CS literature, such as orthogonal matching pursuits  \cite{OMP1993,OMP2003,eladbook2010}, compressed sampling matching pursuits  \cite{Redbook,needell09cosamp}, subspace pursuits \cite{Redbook,DM09},   thresholding algorithms \cite{IHT2004,eladbook2010,Redbook,IHT2012,MZ20}, and the newly developed optimal $k$-thresholding algorithms \cite{Zhao19}.

 Recently, a new framework of reweighted algorithms for   sparse optimization problems was proposed in \cite{zhaom2015,zhao2016,zhaobook2018} which is derived from the  perspective of the dual density.  The  key idea is to use the complementarity between the solutions of the $\ell_0$-minimization and theoretically equivalent weighted $\ell_1$-minimization problem.  Such complementarity property makes it possible to reformulate the $\ell_0$-minimization problem as an equivalent bilevel optimization which seeks the densest solution of the dual problem of a  weighted $\ell_{1}$-problem (see \cite{zhaobook2018} for details). In this paper, we generalize this idea to the  $\ell_{0}$-minimization problem (\ref{Pnew}) and develop new dual-density-based algorithms through convex relaxation of the bilevel optimization. More specifically, to possibly solve the model \eqref{Pnew}, we consider the   problem
\begin{equation}\label{PnewW}
\begin{array}{lll}
 &\min\limits_x &\lbrace \left \| Wx \right \|_{1}=w^{T}\vert x\vert: x\in T\rbrace,
\end{array}
\end{equation}
which is  the  weighted $\ell_{1}$-minimization problem associated with the   problem \eqref{Pnew} for a given weight $w\in R^n_+$ $(W=\mathrm{diag}(w))$.
The dual-density-based reweighted $\ell_{1}$-algorithms for \eqref{Pnew}  are directly derived from the relaxation of the bilevel optimization reformulation of the problem \eqref{Pnew}.  To this goal, we develop a sufficient condition  for the strict complementarity of the solutions of weighted $\ell_1$-minimization problem associated with the problem \eqref{Pnew} and the solutions of its dual problem.  We propose three types of convex relaxations  of the bilevel optimization problem in order to develop our dual-density-based $\ell_1$-algorithms for the  problem \eqref{Pnew}.

The paper is organized as follows. In Sect. \ref{merit function},  we recall the  merit functions for sparsity and  give a few examples of such functions, and we introduce  the classic reweighted $\ell_{1}$-algorithms.   Sect. \ref{chapter weight 1} is denoted  to the development of  a sufficient condition for the strict complementarity property to hold. In Sect. \ref{chapter weight basis}, we show that  the $\ell_0$-problem \eqref{Pnew} can be reformulated equivalently as a bilevel optimization problem which, in theory, can generate  an optimal weight for weighted $\ell_1$-minimization problems.  In Sect. \ref{Dual algorithm}, we discuss several new relaxation strategies for such a bilevel optimization problem, based on which we develop the dual-density-based reweighted $\ell_{1}$-algorithms for the problem \eqref{Pnew}. Finally, we demonstrate some numerical results for the proposed algorithms.

$\bf{Notation:}$ The $\ell_{p}$-norm on $R^{n}$ is defined  as $\left\| x\right\|_{p}=(\sum_{i=1}^{n}\left| x_{i}\right|^{p})^{1/p}$, where $p\geq 1$.   The $n$-dimensional Euclidean space is denoted by $R^{n}$.  $R^{n}_{+}$ and  $R^{n}_{++}$ are the sets of nonnegative and positive vectors respectively. The set of $m \times n$ matrices is denoted by $R^{m\times n}$. The identity matrix of a suitable size is denoted by $I$.
The complementary set of $S\subseteq \left\{ 1,...,n\right\}$ is denoted by $\bar{S}$, i.e., $\bar{S}=\lbrace 1,...,n\rbrace \setminus S$.
 For a given vector $x\in R^{n}$ and $S\subseteq \left\{ 1,...,n\right\}, $    $x_{S} $ is the subvector of $x$ supported on $S.$

 \section{Preliminary}\label{merit function}
In this section, we recall the notion of merit functions for sparsity and list a few such examples. We also briefly outline the classic reweighted $\ell_{1}$-methods for the problem \eqref{Pnew}.  A function is called a merit function for sparsity if it can approximate the $\ell_{0}$-norm  in some senses \cite{zhaobook2018,zhaoli2012}.  Some concave functions are shown to  be the good candidates for the merit functions for sparsity \cite{harikumar1996,CWB2008,zhaoli2012,zhaom2015,zhaobook2018}.
 As pointed out in  \cite{zhaoli2012,zhao2016}, we  may choose  a family of  merit functions in the form
\begin{equation*}\label{add3}
\mathrm{\Psi}_{\varepsilon}( s)=\sum\limits_{i=1}^{n}\varphi_{\varepsilon}( s_{i}),~s\in R^{n}_{+},
\end{equation*} where $ \varphi_{\varepsilon} $ is a function from $R_+ $ to $ R_+.$
$\mathrm{\Psi}_{\varepsilon}( s)$ satisfies  the following properties:
\begin{itemize}[label=$\bullet$]
\item ($P1$) for any given $s\in R_{+}^{n}$, $\mathrm{\Psi}_{\varepsilon}( s)$ tends to $\left\| s\right\|_{0}$ as $\varepsilon$ tends to $0$;
\item($P2$)  $\mathrm{\Psi}_{\varepsilon}( s)$  is twice continuously differentiable   with respect to  $s \in R^n_{+}$ in the open neighborhood of $R^n_{+}; $
\item{($P3$)} $\varphi_{\varepsilon}( s_i)$ is  concave and   strictly increasing with respect to every $s_{i}\in R_{+}$.

\end{itemize}
We denote the set of such merit functions by
 $$\textbf{F}=\lbrace \mathrm{\Psi}_{\varepsilon}:  \mathrm{\Psi}_{\varepsilon} ~  \mathrm{ satisfies }  ~  (P1),(P2)~ \mathrm{and}~(P3) \rbrace. $$
   The following  merit functions  satisfying  $(P1)$-$(P3)$ have been used  in \cite{zhaoli2012,zhao2016}:
\begin{eqnarray}
\mathrm{\Psi}_{\varepsilon}( s) & = & n-\frac{\sum_{i=1}^{n}\log( s_{i}+\varepsilon)}{\log\varepsilon},~s\in R_{+}^{n},     \label{merit function eq4}     \\
\mathrm{\Psi}_{\varepsilon}( s) & =  & \sum_{i=1}^{n}\frac{s_{i}}{ s_{i}+\varepsilon},~s\in R_{+}^{n},      \label{merit function eq1} \\
\mathrm{\Psi}_{\varepsilon}( s) & =  & \sum_{i=1}^{n}( s_{i}+\varepsilon^{1/\varepsilon})^{\varepsilon},~s\in R_{+}^{n}  \label{merit function eq2}
\end{eqnarray}
 where $\varepsilon\in (0,1)$. In this paper, we also consider the following merit function:
\begin{equation}\label{merit function eq3}
\mathrm{\Psi}_{\varepsilon}( s)=\frac{2}{\pi}\sum_{i=1}^{n}\arctan(\frac{ s_{i}}{\varepsilon}), ~s\in R_{+}^{n},
\end{equation}
where $\varepsilon>0$. It is easy to show that \eqref{merit function eq3}   belongs to the set $\textbf{F}$.
\begin{lemma}\label{lemma merit}
The function \eqref{merit function eq3}  satisfies $(P1)$-$(P3)$ on $R^{n}_{+}$.
\end{lemma}
\begin{proof}
Obviously, the function \eqref{merit function eq3} satisfies $(P1)$ and $(P2)$. We now prove that it also satisfies $(P3)$. In $R^{n}_{+}$, note that
$$\nabla\mathrm{\Psi}_{\varepsilon}( s)=\left( \nabla \varphi_{\varepsilon}( s_{1}), \ldots , \nabla \varphi_{\varepsilon}( s_{n})\right)^{T} =\frac{2}{\pi}\biggr( \frac{\varepsilon}{s_{1}^{2}+\varepsilon^{2}}, \ldots, \frac{\varepsilon}{s_{n}^{2}+\varepsilon^{2}}\biggr)^{T},$$
and
$$\nabla^{2} \mathrm{\Psi}_{\varepsilon}( s)=\frac{4}{\pi} \mathrm{diag}\biggr( -\frac{\varepsilon s_{1}}{(s_{1}^{2}+\varepsilon^{2})^{2}}, \ldots ,-\frac{\varepsilon s_{n}}{(s_{n}^{2}+\varepsilon^{2})^{2}}\biggr).$$
Due to $s_{i}\geq 0~\mathrm{and}~\varepsilon>0$, we have
$\nabla \varphi_{\varepsilon}( s_{i})>0$ and $\nabla^{2} \varphi_{ \varepsilon}(s_{i})\leq 0$ for $i=1,...,n$  which implies that $\mathrm{\Psi}_{\varepsilon}(s)$ is concave and strictly  increasing with respect to every entry of $s\in R_{+}^{n}$. Thus \eqref{merit function eq3} satisfies $(P1),(P2)$ and $(P3)$.
\end{proof}
In order to compare the  algorithms proposed in later sections, we briefly introduce the classic reweighted $\ell_1$-method.
Following the idea in \cite{zhaoli2012} and \cite{zhaobook2018},   replacing $\left\| x\right\|_{0}$ with $\mathrm{\Psi}_{\varepsilon}(t)\in \textbf{F}$ leads to   the following approximation  of the problem \eqref{Pnew}:
\begin{equation}\label{relaxation Pnew2}
\min_{(x, t)} \lbrace \mathrm{\Psi}_{\varepsilon}(t):    x\in T, ~\vert x\vert\leq t\rbrace.
\end{equation}
By using the first order approximation of $\mathrm{\Psi}_{\varepsilon}(t)\in \textbf{F}$ at the point $t^{k},$  the problem
 \eqref{relaxation Pnew2} can be approximated by the  optimization
 \begin{equation}\label{relaxation Pnew3}
\min_{(x, t)} \lbrace \nabla \mathrm{\Psi}_{\varepsilon}^{T}(t^{k})t:    x\in T, ~\vert x\vert\leq t\rbrace,
\end{equation}
which  is used to generate the new iterate $(x^{k+1}, t^{k+1}). $
Due to the fact that $\mathrm{\Psi}_{\varepsilon}(t)$ is strictly increasing with respect to each $t_{i}\in R_{+},$  it is evident that the iterate $(x^{k}, t^{k})$ must satisfy $t^{k}=\vert x^{k}\vert$, which implies that
\begin{equation*}\label{relaxation Pnew5}
x^{k+1}\in \argmin_{x} \lbrace \nabla \mathrm{\Psi}_{\varepsilon}^{T}(\vert x^{k}\vert)\vert x\vert: x\in T\rbrace.
\end{equation*}
This is the classic reweighted $\ell_1$-minimization method described in \cite{zhaobook2018}.
 \begin{algorithm}[H]
\caption{ Reweighted  $\ell_{1}$-algorithm (\textbf{RA})}
 \begin{algorithmic}[4]\label{algorithm unified 1}
 \REQUIRE~~\\
 merit function $\mathrm{\Psi}_{\varepsilon}\in \textbf{F}$, matrices $A\in R^{m\times n}$ and $B\in R^{l\times n}$;\\
  vectors $y\in R^{m}$, $b\in R^{l}$, and small positive parameters $(\varepsilon,\epsilon)\in R^{2}_{++}$;\\
     initial weight $w^{0}$,  the iteration index $k$ and the largest number of iterations $k_{\max}$.\\
\ENSURE ~~
At the current iterate $x^{k-1}$, solve the weighted $\ell_{1}$-minimization $$x^{k}\in \argmin \biggr\lbrace\sum\limits_{i=1}^{n}w_{i}^{k} \vert x_{i}\vert:x\in T\biggr\rbrace,$$\\
where $w_{i}^{k} = (\nabla \mathrm{\Psi}_{\varepsilon}(\vert x^{k-1}\vert))_{i}= \nabla \varphi_{\varepsilon}(\vert x^{k-1}_{i}\vert),~i=1,...,n.$
\UPDATE ~~   $w_{i}^{k+1}:=(\nabla \mathrm{\Psi}_{\varepsilon}(\vert x^{k}\vert))_{i}=\nabla \varphi_{\varepsilon}(\vert x^{k}_{i}\vert)$, $i=1,...,n$;
Repeat the above main step until $k=k_{\max}$ (or certain other stopping criterion is met).
\end{algorithmic}
\end{algorithm}
Based on  the generic convergence of revised Frank-Wolfe algorithms ($FW$-$RD$)  for a class of concave functions in \cite{rinaldi2011}, the generic convergence of the algorithm RA can be obtained (see details in \cite{rinaldi2011}), that is, there exists a family of merit functions $\mathrm{\Psi}_{\varepsilon}\in \textbf{F}$  such that  RA  converges to a stationary point of the problem. The convergence of RA to a sparse point in the case of linear-system constraints can be found in \cite{zhaobook2018}.

\section{Duality, strict complementarity and optimality condition}\label{chapter weight 1}
To develop the dual-density-based reweighted $\ell_{1}$-algorithms, we   first  discuss   the duality and the optimality condition  of the  model \eqref{PnewW}, and we  give a sufficient condition for  the strict complementarity to satisfy for  the model \eqref{PnewW}.
\subsection{Duality and complementary condition}

By introducing two variables $t\in R^{n}$ and $\gamma\in R^{m}$ such that $$\vert x \vert\leq t ~\mathrm{and} ~\gamma=y-Ax,$$ we can rewrite \eqref{PnewW} as the following  problem:
\begin{equation}\label{PnewW1}
\begin{array}{lcl}
  & \min\limits_{(x, \gamma, t)} & w^T t\\
& \mathrm{s.t.} & \left \|\gamma \right \|_{2}\leq \epsilon, ~ Bx\leq b,\\
& &\gamma=y-Ax,~ \left | x \right | \leq t,~ t\geq 0.
\end{array}
\end{equation}
Obviously, \eqref{PnewW1} is equivalent to \eqref{PnewW}. Additionally,  if $w\in R^{n}_{++}$, then the solution $(x^{*}, t^{*}, \gamma^{*})$ to \eqref{PnewW1} must satisfy that  $\vert x^{*} \vert= t^{*}$ and $\gamma^{*}=y-Ax^{*}$, and the following relation of the solutions of \eqref{PnewW} and \eqref{PnewW1} is obvious.
\begin{lemma}
If $x^{*}$ is optimal to the problem \eqref{PnewW}, then all vectors $(x^{*}, t^{*}, \gamma^{*})$ satisfying
$$\vert x^{*}_{\mathrm{supp}(w)}\vert=t^{*}_{\mathrm{supp}(w)},~\vert x^{*}_{\overline{\mathrm{supp}(w)}}\vert\leq t^{*}_{\overline{\mathrm{supp}(w)}}~~\mathrm{and}~~ \gamma^{*}=y-Ax^{*}$$
are optimal to the problem \eqref{PnewW1}.
Moreover, if $(\bar{x}, \bar{t}, \bar{\gamma})$ is optimal to the problem \eqref{PnewW1}, then $\bar{x}$ is optimal to the problem \eqref{PnewW}.
\end{lemma}

Let $\lambda=(\lambda_{1},...,\lambda_{6})$ be the dual variable, then the dual problem of (\ref{PnewW1})  can be stated as follows:
\begin{equation}\label{PdW1}
\begin{array}{lcl}
  & \max\limits_{\lambda} & -\lambda_{1}\epsilon-\lambda_{2}^{T}b+\lambda_{3}^{T}y\\
& \mathrm{s.t.} & B^{T}\lambda_{2}-A^{T}\lambda_{3}+\lambda_{4}-\lambda_{5}=0,\\
& &w=\lambda_{4}+\lambda_{5}+\lambda_{6}, ~\left \|\lambda_{3}\right\|_{2}\leq \lambda_{1},\\
& & \lambda_{i}\geq 0,~ i=1,2,4,5,6,.
\end{array}
\end{equation}
The strong duality  between \eqref{PnewW1} and \eqref{PdW1} can be guaranteed under suitable condition. Thus the following results follows from the classic optimization theory \cite{boyd2004book}.
\begin{lemma}
\label{Slater 2}
Let  Slater condition  hold for the convex problem \eqref{PnewW1}, i.e., there exists $(x^{*}, \gamma^{*}, t^{*})\in ri(T)$  such that $$\left\| \gamma^{*}\right\|_{2}< \epsilon, ~Bx^{*}\leq b, ~\vert x^{*}\vert\leq t^{*}, ~y=Ax^{*}+\gamma^{*},~ t^{*}\geq 0,$$
where $ri(T)$ is the relative interior of $T$.
 Then there is no duality gap between \eqref{PnewW1} and its dual problem \eqref{PdW1}. Moreover, if the optimal value of \eqref{PnewW1} is finite, then  there exists at least one optimal Lagrangian multiplier such that the dual optimal value can be  attained.
\end{lemma}
In this paper, we assume that  Slater condition holds for \eqref{PnewW1}. Clearly, the optimal value of \eqref{PnewW1} is finite when $w$ is a given  vector, and hence the  strong duality holds for \eqref{PnewW1} and  \eqref{PdW1} and the dual optimal value can be attained.  Actually, the set $\mathrm{\Omega}=\lbrace x: Ax=y, ~Bx\leq b\rbrace$ is in practice not empty due to the fact that $y$ and $b$ are the measurements of the signals. Thus  Slater condition is a very mild sufficient condition for strong duality to hold for the problems  \eqref{PnewW1} and \eqref{PdW1}.
\subsection{Optimality condition for (\ref{PnewW1}) and (\ref{PdW1}) }\label{subsection optimality condition}
It is well-known that for any convex minimization problem with differentiable objective and constraint functions for which the strong duality holds,   Karush-Kuhn-Tucker (KKT) condition is the necessary and sufficient optimality condition  for the  problem and its dual problem \cite{boyd2004book}. Since  Slater condition holds for \eqref{PnewW1}, by Lemma \ref{Slater 2},  the optimality condition for \eqref{PnewW1} is stated as follows.
\begin{theorem}\label{KKT for Pnew}
If Slater condition holds for  \eqref{PnewW1}, then $(x^{*}, \gamma^{*}, t^{*})$ is optimal to \eqref{PnewW1} and $\lambda_{i}^{*}, i=1,...,6$ is optimal to \eqref{PdW1}  if and only if $(x^{*}, \gamma^{*}, t^{*}, \lambda^{*})$ satisfies the KKT conditions for \eqref{PnewW1}, i.e.,
\begin{equation}\label{cs1}
\left\{\begin{array}{lll}
\gamma^{*}=y-Ax^{*},~ \left\| \gamma^{*}\right\|_{2}\leq \epsilon, ~ x^{*}\leq t^{*}, ~ -t^{*}\leq x^{*}, \\
 Bx^{*}\leq b, ~ t^{*}\geq 0,~ \lambda_{i}^{*}\geq 0, ~i=1,2,4,5,6, \\
  \lambda^{*}_{1}(\epsilon-\left \|\gamma^{*} \right \|_{2})=0,  ~\lambda_{2}^{*T}(b-Bx^{*})=0,  \\
 \lambda_{4}^{*T}(t^{*}-x^{*})=0,  ~\lambda_{5}^{*T}(x^{*}+t^{*})=0, ~\lambda_{6}^{*T}t^{*}=0, \\
  \partial_{x} L(x^{*}, \gamma^{*}, t^{*}, \lambda^{*})=B^{T}\lambda_{2}^{*}-A^{T}\lambda_{3}^{*}+\lambda_{4}^{*}-\lambda_{5}^{*}=0,\\
 \partial_{\gamma} L(x^{*}, \gamma^{*}, t^{*},\lambda^{*})=(\lambda^{*}_{1})\nabla(\left\| \gamma^{*}\right\|_{2})-\lambda^{*}_{3}=0,\\
 \partial_{t} L(x^{*}, \gamma^{*}, t^{*},\lambda^{*})=w-\lambda_{4}^{*}-\lambda_{5}^{*}-\lambda_{6}^{*}=0.\\
\end{array} \right.
\end{equation}
where  $L(x^{*}, \gamma^{*}, t^{*}, \lambda^{*})=w^{T}t^{*}-\lambda^{*}_{1}(\epsilon-\left \|\gamma^{*} \right \|_{2})-
 \lambda_{2}^{*T}(b-Bx^{*}) -\lambda_{3}^{*T}(Ax^{*}+\gamma^{*}-y) -\lambda_{4}^{*T}(t^{*}-x^{*})-\lambda_{5}^{*T}(x^{*}+t^{*})-\lambda_{6}^{*T}t^{*}$.
\end{theorem}
\noindent From the optimality condition   in \eqref{cs1},  we see that  $t^{*}$ and $\lambda^{*}_{6}$ satisfy the complementary condition.
\begin{corollary}\label{cs theorm}
Let  Slater condition hold for \eqref{PnewW1}.  Then, for  any optimal solution pair  $((x^{*}, t^{*}, \gamma^{*}),\lambda^{*})$, where $(x^{*}, t^{*}, \gamma^{*})$ is optimal to \eqref{PnewW1} and $\lambda^{*}=(\lambda_{1}^{*},..., \lambda_{6}^{*})$ is optimal to \eqref{PdW1},   $t^{*}$ and $\lambda_{6}^{*}$ are complementary in the sense that
$$(t^{*})^{T}\lambda_{6}^{*}=0, ~  t^{*}\geq 0 ~ \mathrm{and} ~ \lambda_{6}^{*}\geq 0.$$
\end{corollary}
\noindent Clearly, if  $(x^{*}, t^{*}, \gamma^{*})$ is optimal to \eqref{PnewW1} and $w$ is positive, it must hold $\vert x^{*}\vert=t^{*}$. Hence by Corollary \ref{cs theorm}, for $i=1,...,n$, we have
\begin{equation}\label{x t transfer}
\vert x^{*}_{i}\vert(\lambda_{6}^{*})_{i}=0, ~ (\lambda_{6}^{*})_{i}\geq 0.
\end{equation}
When $w$ is nonnegative, and if $(x^{*}, t^{*}, \gamma^{*})$ is optimal to \eqref{PnewW1}, we have
$$\vert x^{*}_{i}\vert=t_{i}^{*},~ i\in \mathrm{supp}(w);~\vert x^{*}_{i}\vert\leq t_{i}^{*}, ~i\in \overline{\mathrm{supp}(w)}.$$

\noindent For $i\in \mathrm{supp}(w)$, \eqref{x t transfer} is  valid. For $i\in \overline{\mathrm{supp}(w)}$, due to the constraints $w=\lambda_{4}+\lambda_{5}+\lambda_{6}$ and $\lambda_{4}, \lambda_{5}, \lambda_{6} \geq 0$, $w_{i}=0$ implies that $(\lambda_{6}^{*})_{i}=0$. This means \eqref{x t transfer} is also valid for $i\in \overline{\mathrm{supp}(w)}$. Therefore, we have the following result:
\begin{theorem}\label{cs theorem 1}
Let $w$ be a nonnegative given vector, and let  Slater condition hold for \eqref{PnewW1}. Then, for  any optimal solution pair  $((x^{*}, t^{*}, \gamma^{*}),\lambda^{*})$, where $(x^{*}, t^{*}, \gamma^{*})$ is optimal to \eqref{PnewW1} and $\lambda^{*}= (\lambda_{1}^{*},..., \lambda_{6}^{*})$ is optimal to \eqref{PdW1},   $\vert x_{i}^{*}\vert$ and $(\lambda_{6}^{*})_{i}$ are complementary in the sense that
\begin{equation}\label{cs x and lamda}
\vert x_{i}^{*}\vert(\lambda_{6}^{*})_{i}=0 ~ \mathrm{and} ~ (\lambda_{6}^{*})_{i}\geq 0, ~i=1,...,n.
\end{equation}
\end{theorem}
The relation \eqref{cs x and lamda} implies that
$$ \left\| x^{*} \right\|_{0}+\left\| \lambda_{6}^{*} \right\|_{0}\leq n, $$
where $n$ is   the dimension of $x^{*}$ or $\lambda_{6}^{*}$.
Suppose $\vert x^{*}\vert $ and $\lambda^{*}_{6}$ are strictly complementary, i.e.,
$$\vert x^{*}\vert^{T}\lambda_{6}^{*}=0,  ~\lambda_{6}^{*}\geq 0 ~ \mathrm{and}~\vert x^{*} \vert+\lambda_{6}^{*}> 0.$$ Then
$$\left\| x^{*} \right\|_{0}+\left\| \lambda_{6}^{*} \right\|_{0}= n.$$

\subsection{Strict complementarity}
 For  nonlinear optimization models, the strictly complementary   property might not hold. However,  it might be possible to develop a condition such that the strict complementarity  holds for the model \eqref{PnewW} or \eqref{PnewW1}.  We now develop such a condition for the problems \eqref{PnewW1} and \eqref{PdW1}  under  the following assumption.
\begin{assumption}\label{assumption 1}
Let $W=\mathrm{diag}(w)$ satisfy the following properties:
\begin{itemize}[label=$\bullet$]
\item{$\langle G1\rangle$}\quad The problem \eqref{PnewW} with $w$ has an optimal solution which is a relative interior point in the feasible set $T$, denoted by $x^{*}\in ri(T)$, such that $$\left\| y-Ax^{*}\right\|_{2}<\epsilon, ~Bx^{*}\leq b,$$
\item{$\langle G2\rangle$}\quad the optimal value $Z^{*}$ of  \eqref{PnewW} is finite and positive, i.e., $Z^{*}\in (0, \infty)$,\\
\item{$\langle G3\rangle$}\quad $w_{j}\in (0, \infty]$ for all $1\leq j \leq n$.
\end{itemize}
\end{assumption}
\begin{example}\label{example 1}
Consider the  system $\left\| y-Ax\right\|_{2}\leq \epsilon, Bx\leq b$ with $\epsilon=10^{-1}$, where
{\small $$ A=\left[ \begin{array}{cccc}
1 & 0 & -2 &5 \\
 0&  1&  4& -9\\
 1&  0&  -2&5
\end{array} \right], B=\left[ \begin{array}{cccc}
-0.5 & 0 & 1 &-2.5 \\
 0.5&  -0.5&  -1& 2\\
 -3&  -3&  -2&3
\end{array} \right], y=\left[ \begin{array}{c}
1\\
-1\\
1
\end{array} \right],b=\left[ \begin{array}{c}
-0.5\\
1\\
-1
\end{array} \right].$$}
\end{example}
\noindent We can see that the problem \eqref{PnewW} with $w=(1,100,1,100)^{T}$ has an optimal solution $(1/2,0,-1/4,0)^{T}$ which satisfies Assumption \ref{assumption 1}.

  Next  we prove the following theorem concerning  the strict complementarity for \eqref{PnewW1} and \eqref{PdW1} under Assumption \ref{assumption 1}.

\begin{theorem}\label{scc}
 Let  $y$ and $b $ be two given vectors, $A\in R^{m\times n}$ and $B\in R^{l\times n}$ be two given matrices, and $w$ be a given weight which satisfies  Assumption \ref{assumption 1}.
Then there exists a pair $((x^{*}, t^{*}, \gamma^{*}), \lambda^{*})$, where $(x^{*}, t^{*}, \gamma^{*})$ is an optimal solution to \eqref{PnewW1} and $\lambda^{*}=(\lambda_{1}^{*},...,\lambda_{6}^{*})$ is an optimal solution to \eqref{PdW1}, such that $t^{*}$ and $\lambda_{6}^{*}$ are strictly complementary, i.e.,
$$(t^{*})^{T}\lambda_{6}^{*}=0, ~t^{*}+\lambda_{6}^{*}>0, ~(t^{*}, \lambda_{6}^{*})\geq 0.$$
\end{theorem}

\begin{proof}
Note that $(G1)$ in Assumption \ref{assumption 1} implies that  Slater condition holds for \eqref{PnewW1}. This, combined with $(G2)$, indicates  from Lemma \ref{Slater 2} that the duality gap is $0$, and  the optimal value $Z^{*}$ for \eqref{PdW1} can be attained.   For any given index $j: 1\leq j \leq n$, we consider a series of  minimization problems:
\begin{equation}\label{proof eq1}
\begin{array}{cl}
\min\limits_{(x, t, \gamma)}& -t_{j}\\
  \mathrm{s.t.}&\left \|\gamma \right \|_{2}\leq \epsilon,~Bx\leq b,~ \gamma=y-Ax,\\
 &\left | x \right | \leq t, ~-w^{T}t\geq -Z^{*}, ~ t\geq 0.
\end{array}
\end{equation}
 The dual problem of \eqref{proof eq1} can be obtained by using the same method for developing the dual problem of \eqref{PnewW1}, which is stated as follows:
\begin{equation}\label{proof eq2}
\begin{array}{lcl}
&\max\limits_{(\mu, \tau)}&-\mu_{1}\epsilon-\mu_{2}^{T}b+\mu_{3}^{T}y-\tau Z^{*}\\
 &\mathrm{s.t.}& B^{T}\mu_{2}-A^{T}\mu_{3}+\mu_{4}-\mu_{5}=0,~\left \|\mu_{3}\right\|_{2}\leq \mu_{1},\\
 & &\tau w=\mu_{4}+\mu_{5}+\mu_{6}+p,~ \mu_{i}\geq 0,~ i=1,2,4,5,6, ~\tau\geq 0,
\end{array}
\end{equation}
where $p$ is a vector whose $j$th component is $1$ and the remains  are $0$, i.e., $$p_{j}=1;~p_{i}=0,~i\neq j.$$ Next we  show that  \eqref{proof eq1} and \eqref{proof eq2} satisfy the strong duality property under  Assumption \ref{assumption 1}.  It can be seen that $(x,t,\gamma)$ is a feasible solution to \eqref{proof eq1} if and only if  $(x,t,\gamma)$ is an optimal solution of \eqref{PnewW1}, or if  $x$ is optimal to \eqref{PnewW}. If $w$ satisfies the conditions in Assumption \ref{assumption 1}, then there exists an optimal solution $\bar{x}$ of \eqref{PnewW} such that
$\left\|y-A\bar{x}\right\|_{2}<\epsilon,~B\bar{x}\leq b $ and $w^{T}\vert \bar{x}\vert=Z^{*},$  which means  there is a relative interior point $(\bar{x}, \bar{t}, \bar{\gamma})$ of the feasible set of \eqref{proof eq1}   satisfying
$$\left\| \bar{\gamma} \right\|_{2}<\epsilon,~ B\bar{x}\leq b,~ \bar{\gamma}=y-A\bar{x}, ~\vert \bar{x} \vert \leq \bar{t}, ~w^{T}\bar{t}\leq Z^{*}, ~\bar{t}\geq 0.$$
 As a result, the strong duality  holds for \eqref{proof eq1} and \eqref{proof eq2} for all $j$. Moreover, due to $(G2)$ and $(G3)$,  $w$ is  positive and $Z^{*}$ is finite,  so $t_{j}$ cannot be $\infty$.  Thus the optimal value of all $j$th minimization problems \eqref{proof eq1} is finite. It follows from Lemma  \ref{Slater 2} that for each $j$th optimization  \eqref{proof eq1} and \eqref{proof eq2}, the duality gap is $0$, and each $j$th dual problem \eqref{proof eq2} can achieve their optimal value.

We use $\xi^{*}_{j}$ to denote the optimal value of the $j$th  problem in \eqref{proof eq1}.  Clearly, $\xi^{*}_{j}$ is nonpositive, i.e.,  $$\quad  \xi^{*}_{j}<0 ~~ \mathrm{or} ~~ \xi^{*}_{j}=0.$$

Case 1: \quad  $\xi^{*}_{j}<0$. Then \eqref{PnewW1} has  an optimal solution $(x', t', \gamma')$ where the $j$th component in $t'$  is positive since $t_{j}'=-\xi^{*}_{j}$ and admits the largest value amongst all the optimal solutions of \eqref{PnewW1}. By Theorem \ref{cs theorm}, the complementary condition implies that \eqref{PdW1}  has an optimal solution $\lambda'=(\lambda'_{1},...,\lambda'_{6})$ where $j$th component in $\lambda_{6}'$  is 0.  Then we have an optimal solution pair $((x',t',\gamma'),\lambda')$ for  \eqref{PnewW1} and  \eqref{PdW1} such that $t_{j}'>0$ and $(\lambda_{6}')_{j}=0$. It means that  $$t_{j}'=-\xi^{*}_{j}>0  \quad \mathrm{implies} \quad ( \lambda_{6}')_{j}=0.$$

Case 2: \quad   $\xi^{*}_{j}=0$. Following from  the strong duality between \eqref{proof eq1} and \eqref{proof eq2}, we have an optimal solution $(\mu, \tau )$ of the $j$th optimization problem \eqref{proof eq2} such that
$$-\mu_{1}\epsilon-\mu_{2}^{T}b+\mu_{3}^{T}y=\tau Z^{*}.$$  First, we consider $\tau\neq 0$. The above equality can be reduced to
$$-\frac{\mu_{1}\epsilon}{\tau}-\frac{\mu_{2}^{T}}{\tau}b+\frac{\mu_{3}^{T}}{\tau}y=Z^{*},$$
and we also have
$$B^{T}\frac{\mu_{2}}{\tau}-A^{T}\frac{\mu_{3}}{\tau}+\frac{\mu_{4}}{\tau}-\frac{\mu_{5}}{\tau}=0,~\left\| \frac{\mu_{3}}{\tau}\right\|_{2}\leq \frac{\mu_{1}}{\tau}, ~
 w=\frac{\mu_{4}}{\tau}+\frac{\mu_{5}}{\tau}+\frac{\mu_{6}}{\tau}+\frac{p}{\tau}.$$
We set $$\lambda_{1}^{'}=\frac{\mu_{1}}{\tau}, ~\lambda_{2}^{'}=\frac{\mu_{2}}{\tau}, ~\lambda_{3}^{'}=\frac{\mu_{3}}{\tau}, ~\lambda_{4}^{'}=\frac{\mu_{4}}{\tau}, ~\lambda_{5}^{'}=\frac{\mu_{5}}{\tau}, ~\lambda_{6}^{'}=\frac{\mu_{6}}{\tau}+\frac{p}{\tau}.$$  Due to  strong duality of \eqref{PnewW1} and \eqref{PdW1} again,   $\lambda^{'}=( \lambda_{1}^{'},..., \lambda_{6}^{'})$ is optimal to \eqref{PdW1}. Note that
$$(\lambda_{6})_{j}^{'}=\frac{(\mu_{6})_{j}+1}{\tau}.$$
Thus  $(\lambda_{6})_{j}^{'}>0$, which follows from $\mu_{6}\geq 0$ and $\tau>0$. Thus $$t'_{j}=-\xi_{j}^{*}=0~ ~\mathrm{implies} ~~
(\lambda_{6})_{j}^{'}>0.$$

Note that the third constraint in $j$th optimization of \eqref{proof eq2} requires $\tau\neq 0$ since $w$, $\mu_{4}$, $\mu_{5}$, $\mu_{6}$ are all non-negative and  $p_{j}=1$ so that the $j$th component in $\tau w$ must be greater or equal than 1. Therefore, all  $j$th optimization problems in \eqref{proof eq2} are infeasible if $\tau=0$. As a result, the optimal solution $(\mu, \tau)$ of \eqref{proof eq2} with $\tau=0$ is impossible to occur. Combining the cases 1 and 2 implies that  for each $1\leq j \leq n$, we have an optimal solution pair $((x^{(j)}, t^{(j)}, \gamma^{(j)}), \lambda^{(j)})$ such that  $t^{(j)}_{j}>0$ or $(\lambda_{6}^{(j)})_{j}>0$. For all $j$th solution pairs,  they all satisfy the following properties:
\begin{itemize}[label=$\bullet$]
\item[(i)]    $(x^{(j)}, t^{(j)}, \gamma^{(j)})$ is optimal to \eqref{PnewW1}, and  $(\lambda_{1}^{(j)}, \lambda_{2}^{(j)}, \lambda_{3}^{(j)}, \lambda_{4}^{(j)}, \lambda_{5}^{(j)}, \lambda_{6}^{(j)})$ is optimal to \eqref{PdW1};
\item[(ii)]     the $j$th component of $t^{(j)}$ and the $j$th component of $\lambda_{6}^{(j)}$ are strictly complementary, such that $t^{(j)}_{j}(\lambda_{6}^{(j)})_{j}=0,~ t^{(j)}_{j}+(\lambda_{6}^{(j)})_{j} >0 $.
\end{itemize}
Denote $(x^{*}, t^{*}, \gamma^{*},\lambda^{*})$ by
$$x^{*}=\frac{1}{n}\sum_{j=1}^{n} x^{(j)}, ~  t^{*}=\frac{1}{n}\sum_{j=1}^{n}  t^{(j)}, ~ \gamma^{*}=\frac{1}{n}\sum_{j=1}^{n}\gamma^{(j)}, ~ \lambda_{i}^{*}=\frac{1}{n}\sum_{j=1}^{n}\lambda^{(j)}_{i}, ~ i=1,2,\cdots,6.$$
Since $(x^{(j)}, t^{(j)}, \gamma^{(j)}),~ j=1,2,...,n$ are all  optimal solutions of \eqref{PnewW1}, then for any $ j$, we have
\begin{equation}\label{proof eq3}
\left\{\begin{array}{ll}
w^{T}t^{(j)}=Z^{*},~\left\| \gamma^{(j)} \right\|_{2}\leq \epsilon,~ Bx^{(j)}\leq b, \\
\gamma^{(j)}=y-Ax^{(j)}, ~ \vert x^{(j)}\vert\leq t^{(j)}, ~t^{(j)}\geq 0.
\end{array} \right.
\end{equation}
It is easy to see that $$w^{T}t^{*}=Z^{*},~ Bx^{*}\leq b, ~\gamma^{*}=y-Ax^{*}, ~ t^{*}\geq 0.$$
Moreover,
$$\left\|\gamma^{*}\right\|_{2}=  \left\| \frac{1}{n}\sum_{j=1}^{n}\gamma^{(j)}\right\|_{2} \leq \sum_{j=1}^{n}\left\| \frac{1}{n}\gamma^{(j)}\right\|_{2}\leq \epsilon,$$
$$\vert x^{*}\vert=\left| \frac{1}{n}\sum_{j=1}^{n}x^{(j)}\right| \leq \frac{1}{n}\sum_{j=1}^{n}\vert x^{(j)}\vert\leq \frac{1}{n}\sum_{j=1}^{n}t^{(j)}=t^{*},$$
where the first inequality of each  equation above follows from the triangle inequality. Then the vector $(x^{*}, t^{*}, \gamma^{*})$ satisfies
 \begin{equation}\label{proof eq4}
\left\{\begin{array}{ll}
w^{T}t^{*}=Z^{*},~\left\| \gamma^{*} \right\|_{2}\leq \epsilon, ~ Bx^{*}\leq b,\\
 \gamma^{*}=y-Ax^{*}, ~\vert x^{*}\vert\leq t^{*}, ~t^{*}\geq 0.
\end{array} \right.
\end{equation}
Thus $(x^{*}, t^{*}, \gamma^{*})$ is optimal to \eqref{PnewW1}, and similarly it can be proven that  $\lambda^{*}=(\lambda^{*}_{1},...,\lambda^{*}_{6})$ is an optimal solution to \eqref{PdW1}. By strong duality,  $t^{*}$ and $\lambda_{6}^{*}$ are complementary. Due to the above-mentioned property (2),  it is impossible to find a pair $(t^{*}, \lambda_{6}^{*})$ such that their $j$th components  are both 0. Thus,  $(t^{*}, \lambda^{*}_{6})$ is the strictly complementary solution pair for \eqref{PnewW1} and \eqref{PdW1}.
\end{proof}

\begin{remark}
It can be seen that the following two sets
$$P^{*}=\lbrace i: t_{i}^{*}>0\rbrace ~ \mathrm{and}~ Q^{*}=\lbrace i: (\lambda_{6}^{*})_{i}>0 \rbrace$$
are invariant for all pairs of strictly complementary solutions.   Suppose there are two  distinct optimal pairs of the solutions of \eqref{PnewW1} and \eqref{PdW1}, denoted by $(x_{(k)}$, $t_{(k)}$, $\gamma_{(k)}$, $\lambda_{(k)})$,  $k=1,2$, such that  $( t_{(k)},  \lambda_{6(k)}),  k=1,2$ are strictly complementary pairs, where $(x_{(k)}, t_{(k)}, \gamma_{(k)})$ are optimal to \eqref{PnewW1} and $(\lambda_{(k)})$ are optimal to   \eqref{PdW1}.
Due to  Theorem \ref{cs theorm}, we know that
$$(\lambda_{6(1)})^{T}t_{(2)}=0~\mathrm{and}~(\lambda_{6(2)})^{T}t_{(1)}=0.$$
It means that the supports of all  strictly complementary pairs of  \eqref{PnewW1} and  \eqref{PdW1} are invariant. Otherwise, there exists an index $j$ such that $(t_{(1)})_{j}>0$ and $(\lambda_{6(2)})_{j}>0$, leading to a contradiction.
\end{remark}

 Since  the optimal solution $(x^{*}, t^{*}, \gamma^{*})$ to \eqref{PnewW1} must have $t^{*}=\vert x^{*}\vert$ if $w>0$,  the main results of  Theorem \ref{scc} also imply that  $\vert x^{*}\vert$ and $ \lambda_{6}^{*}$ are strictly complementary under  Assumption \ref{assumption 1}.

\section{Bilevel model for optimal weights}\label{chapter weight basis}

For weighted $\ell_{1}$-minimization,  how to determine a weight to guarantee   the exact  recovery,   sign recovery or  support recovery of  sparse signals is an important issue in  CS theory.
Based on the complementary condition and strict complementarity discussed above, we may develop a bilevel optimization model for such a weight, which is called the optimal weight in  \cite{zhaom2015}, \cite{zhao2016} and \cite{zhaobook2018}.

\begin{definition}[Optimal Weight]\label{definition optimal weight}
A weight is called an optimal weight if the solution of the weighted $\ell_{1}$-problem with this weight is one of the optimal solution of the $\ell_{0}$-minimization problem.
\end{definition}

Let  $Z^{*}$ be the optimal value of \eqref{PnewW}. Notice that the optimal solution of \eqref{PnewW} remains the same when $w$ is replaced by $\alpha w$ for any positive $\alpha$. When $Z^{*}\neq 0$, by replacing $W$ by $W/Z^{*}$, we can obtain
$$1=\min_{x}\lbrace \left\| (W/Z^{*}) x\right\|_{1}:x\in T\rbrace,$$
where  $W=\mathrm{diag}(w)$.
We use $\zeta$ to denote the set of such weights, i.e.,
\begin{equation}\label{zeta}
\zeta=\lbrace w\in R^{n}_{+}: ~1=\min_{x}\lbrace \left\| Wx\right\|_{1}, x\in T\rbrace \rbrace.
\end{equation}
Clearly, $\bigcup\limits_{\alpha> 0}\alpha\zeta$ is the set of weights such that \eqref{PnewW} has a finite and positive  optimal value, and $\zeta$ is not necessarily bounded.  Under  Slater condition,  Theorem \ref{cs theorem 1} implies that given any $w\in \zeta$, any optimal solutions of \eqref{PnewW1} and \eqref{PdW1}, denoted by  $(x^{*}(w),t^{*}(w),\gamma^{*}(w))$ and $\lambda^{*}(w)=(\lambda^{*}_{1}(w),..., \lambda^{*}_{6}(w))$, satisfy that   $\vert x^{*}(w)\vert$ and $\lambda_{6}^{*}(w)$ are complementary, i.e.,
 \begin{equation}\label{cc eq1}
  \left\|  x^{*}(w)\right\|_{0}+\left\|  \lambda_{6}^{*}(w)\right\|_{0}\leq n.
 \end{equation}
  If $w^{*}$ satisfies Assumption \ref{assumption 1}, then Slater condition is automatically satisfied for \eqref{PnewW1} with $w^{*}$ and \eqref{cc eq1}  is also valid. Moreover, by Theorem \ref{scc}, there exists a strictly complementary pair $(\vert x^{*}(w^{*})\vert, \lambda_{6}^{*}(w^{*}))$ such that
 $$\left\|  x^{*}(w^{*})\right\|_{0}+\left\|  \lambda_{6}^{*}(w^{*})\right\|_{0}=n.$$
 If $w^{*}$ is an optimal weight (see Definition \ref{definition optimal weight}), then $\lambda^{*}_{6}(w^{*})$ must be  the densest slack variable among all $w\in \zeta$, and locating a sparse vector can be converted to  $$\lambda^{*}_{6}(w^{*})=\argmax\lbrace \left\| \lambda^{*}_{6}(w)\right\|_{0}:w\in \zeta\rbrace.$$
Inspired by the above fact, we  develop a theorem under  Assumption \ref{assumption 2}  which claims that finding a sparsest point in $T$ is equivalent to seeking the proper weight $w$ such that  the dual problem  \eqref{PdW1} has the densest optimal variable $\lambda_{6}$. Such  weights are optimal weights and can be determined by certain bilevel optimization.  This idea was first introduced by Zhao and Ko{\v{c}}vara \cite{zhaom2015} (and also by Zhao and Luo \cite{zhao2016}) to solve  the standard $\ell_{0}$-minimization (C1). In this paper, we generalize their idea to solve the model \eqref{Pnew} by developing new convex relaxation technique for the underlying bilevel optimization problem. Before that we make the following assumption:
\begin{assumption}\label{assumption 2}
Let $\nu$ be an arbitrary sparsest point in $T$ given in \eqref{feasible set T}. There exists a weight $\bar{w}\geq 0$ such that
\begin{itemize}[label=$\bullet$]
\item{$\langle H1\rangle$} the problem \eqref{PnewW} with $\bar{w}$ has an optimal solution $\bar{x}$ such that $\left\| \bar{x}\right\|_{0}=\left\| \nu\right\|_{0}$,
 \item{$\langle H2\rangle$} there exists  an  optimal variable in \eqref{PdW1} with $\bar{w}$, denoted as $\bar{\lambda}$, such that $\bar{\lambda_{6}}$ and $\bar{x}$ are strictly complementary,
 \item{$\langle H3\rangle$} the optimal value of \eqref{PnewW} with $\bar{w}$ is finite and positive.
\end{itemize}
\end{assumption}

 An example for the existence of a weight satisfying Assumption \ref{assumption 2} is given in the remark  following the next theorem.
\begin{theorem}\label{th idea}
Let Slater condition and  Assumption \ref{assumption 2} hold. Consider the bilevel optimization

\begin{equation}\label{belevel1}
\begin{array}{lcl}
& \max\limits_{(w, \lambda)} & \left\| \lambda_{6}\right\|_{0}\\
& \mathrm{s.t.} & B^{T}\lambda_{2}-A^{T}\lambda_{3}+\lambda_{4}-\lambda_{5}=0,~\left \|\lambda_{3}\right\|_{2}\leq \lambda_{1}, \\
& & -\lambda_{1}\epsilon-\lambda_{2}^{T}b+\lambda_{3}^{T}y=\min\limits_{x}\lbrace \left\| Wx\right\|_{1}: x\in T\rbrace,\\
& &w= \lambda_{4}+\lambda_{5}+\lambda_{6}\geq 0,~\lambda_{i}\geq 0, ~ i=1,2,4,5,6,
\end{array}
\end{equation}
where $W=\mathrm{diag}(w)$, and $T$ is given as \eqref{feasible set T}. If $(w^{*}, \lambda^{*})$ is an optimal solution to the above optimization problem \eqref{belevel1}, then  any optimal solution $x^{*}$ to
\begin{equation}\label{original}
\min_{x}\lbrace \left\| W^{*}x\right\|_{1}: x\in T\rbrace,
\end{equation}
 is a sparsest point in $T$, where $W^{*}=\mathrm{diag}(w^{*})$.  
\end{theorem}

\begin{proof}
Let $\nu$ be a sparsest point in $T$.  Suppose that   $(w^{*}, \lambda^{*})$  is an optimal solution of  \eqref{belevel1}. We now prove that any optimal solution to \eqref{original} is a sparsest point in $T$ under Assumption \ref{assumption 2}. Let $w'$ be a weight satisfying  Assumption \ref{assumption 2}. This means that  \eqref{PnewW} with $W=\mathrm{diag}(w')$ has an optimal solution $x'$ such that $\left\| x'\right\|_{0}=\left\| \nu \right\|_{0}$. Moreover,  there exists a strictly complementary pair  ($x', \lambda_{6}'$) satisfying
\begin{equation}\label{idea eq1}
\left\| x'\right\|_{0}+\left\| \lambda_{6}'\right\|_{0}=n= \left\| \lambda_{6}'\right\|_{0}+\left\| \nu\right\|_{0},
\end{equation}
where  $\lambda'=(\lambda_{1}',...,\lambda_{6}')$ is the dual optimal solution of \eqref{PdW1} with $w=w'$, i.e.,
\begin{equation}\label{idea eq5}
\begin{array}{lcl}
& \max\limits_{\lambda}&-\lambda_{1}\epsilon-\lambda_{2}^{T}b+\lambda_{3}^{T}y\\
 &\mathrm{s.t.}& B^{T}\lambda_{2}-A^{T}\lambda_{3}+\lambda_{4}-\lambda_{5}=0,~\left \|\lambda_{3}\right\|_{2}\leq \lambda_{1},\\
 & &w'=\lambda_{4}+\lambda_{5}+\lambda_{6},~ \lambda_{i}\geq 0,~ i=1,2,4,5,6.
\end{array}
\end{equation}
By Lemma \ref{Slater 2},   Slater condition implies that strong duality holds for the problems \eqref{idea eq5} and  \eqref{PnewW1}  with $w'$. Note that  the optimal values of \eqref{PnewW1} and \eqref{PnewW} with $w'$ are equal and finite so that $(w', \lambda')$ is feasible to \eqref{belevel1}. Let $x^{*}$ be an arbitrary solution to \eqref{original}.
Note that   \eqref{PnewW1} with $w^{*}$ is equivalent to  \eqref{original},
to which the dual problem is
\begin{equation}\label{idea eq4}
\begin{array}{lcl}
& \max\limits_{\lambda}&-\lambda_{1}\epsilon-\lambda_{2}^{T}b+\lambda_{3}^{T}y\\
 &\mathrm{s.t.}& B^{T}\lambda_{2}-A^{T}\lambda_{3}+\lambda_{4}-\lambda_{5}=0,~\left \|\lambda_{3}\right\|_{2}\leq \lambda_{1},\\
 & &w^{*}=\lambda_{4}+\lambda_{5}+\lambda_{6},~ \lambda_{i}\geq 0,~ i=1,2,4,5,6.
\end{array}
\end{equation}
 Moreover,   $\lambda^{*}=(\lambda_{1}^{*},...,\lambda_{6}^{*})$ is feasible to  \eqref{idea eq4} and the third constraint of \eqref{belevel1} implies that there is no duality gap between \eqref{PnewW1} with $w^{*}$ and \eqref{idea eq4}. Thus, by strong duality,    $\lambda^{*}=(\lambda_{1}^{*},...,\lambda_{6}^{*})$ is an optimal solution to \eqref{idea eq4}. Therefore, by Theorem \ref{cs theorem 1}, $\vert x^{*}\vert$ and $\lambda_{6}^{*}$ are complementary. Hence, we have
\begin{equation}\label{idea eq2}
\left\| x^{*}\right\|_{0}\leq n- \left\| \lambda_{6}^{*}\right\|_{0}.
\end{equation}
 Since $(w^{*}, \lambda^{*})$ is optimal to  \eqref{belevel1},  we have
\begin{equation}\label{idea eq3}
\left\| \lambda_{6}'\right\|_{0}\leq \left\| \lambda_{6}^{*}\right\|_{0}.
\end{equation}
Plugging \eqref{idea eq1}  and \eqref{idea eq3} into \eqref{idea eq2}  yields
$$\left\| x^{*}\right\|_{0}\leq n- \left\| \lambda_{6}^{*}\right\|_{0} \leq n-\left\| \lambda_{6}'\right\|_{0}=\left\| x'\right\|_{0}=\left\| \nu\right\|_{0},$$
which implies $\left\| x^{*}\right\|_{0}=\left\| \nu\right\|_{0},$ due to the assumption that $\nu$ is the sparsest point in $T$. Then any optimal solution to \eqref{idea eq1} is a sparsest point in $T$.

\end{proof}

Given Assumption \ref{assumption 2} and Slater condition,  finding a sparsest point in $T$ is tantamountly equal to look for the densest dual solution via the bilevel model \eqref{belevel1}.

By the definition of optimal weights, Theorem \ref{th idea}  implies that $w^{*}$ is an optimal weight  by which a sparsest point can be obtained via \eqref{PnewW}.  If there is no weight satisfying the properties in Assumption \ref{assumption 2},  a heuristic method for finding a sparse point in $T$ can be also developed from \eqref{cc eq1}  since  the increase in  $\Vert \lambda_{6}(w)\Vert_{0}$ leads to the decrease of $\left\| x(w) \right\|_{0}$ to a certain level. Before we close this section, we make  some remarks for Assumption \ref{assumption 2}.

\begin{remark}
Consider Example \ref{example 1}.
It can be seen that $(0,0,2,1)^{T}$
 is a sparsest point in the feasible set $T$ of this example. If we choose weight $w=(100, 100, 1, 1)^{T}$, then we can see that $(0, 0, 2, 1)^{T}$ is the unique optimal solution of \eqref{PnewW} which satisfies $\langle H1\rangle$ and $\langle H3\rangle$ in Assumption \ref{assumption 2}.   In addition, $(0, 0, 2, 1)^{T}$ is a relative  interior point in the feasible set $T$. This, combined with the fact that weights are positive, implies that Assumption \ref{assumption 1}  is satisfied, and hence the strict complementarity  is satisfied which means that $\langle H2\rangle$ in Assumption \ref{assumption 2} is satisfied. Specifically, we can find an optimal dual solution $\bar{\lambda}=(\bar{\lambda}_{1},...,\bar{\lambda}_{6})$ with  $\bar{\lambda}_{6}=(32.27, 31.71, 0, 0)^{T}$.  Therefore, in this example, the weight $w=(100, 100, 1, 1)^{T}$ satisfies  Assumption \ref{assumption 2}.\\
\end{remark}

 \section{Dual-density-based algorithms}\label{Dual algorithm}
 Note that it is  difficult to solve  a bilevel optimization. We now  develop three types of relaxation models for solving the bilevel optimization \eqref{belevel1}.
\subsection{Relaxation models}\label{relaxation section1}
Zhao and Luo \cite{zhao2016} presented a method to relax a bilevel problem similar to  \eqref{belevel1}. Motivated by their idea,  we now relax our bilevel model. We focus on  relaxing the difficult constraint $-\lambda_{1}\epsilon-\lambda_{2}^{T}b+\lambda_{3}^{T}y=\min_{x}\lbrace \left\| Wx\right\|_{1}: x\in T\rbrace$  in \eqref{belevel1}.
 By replacing the objective function $\Vert \lambda_{6}\Vert_{0}$ in \eqref{belevel1} by $\mathrm{\Psi}_{\varepsilon}(\lambda_{6})\in \textbf{F}, $ where $ \lambda_{6}\geq 0$, we obtain an approximation problem of \eqref{belevel1}, i.e.,
 \begin{equation}\label{belevel3}
\begin{array}{lcl}
&\max\limits_{(w, \lambda)}& \mathrm{\Psi}_{\varepsilon}(\lambda_{6})\\
 &\mathrm{s.t.}& B^{T}\lambda_{2}-A^{T}\lambda_{3}+\lambda_{4}-\lambda_{5}=0,~\left \|\lambda_{3}\right\|_{2}\leq \lambda_{1}\\
 & & -\lambda_{1}\epsilon-\lambda_{2}^{T}b+\lambda_{3}^{T}y=\min_{x}\lbrace \left\| Wx\right\|_{1}: x\in T\rbrace,\\
 & & w= \lambda_{4}+\lambda_{5}+\lambda_{6} \geq 0,~ \lambda_{i}\geq 0,~ i=1,2,4,5,6.
\end{array}
\end{equation}
 We recall the set of the weights $\zeta$ given in \eqref{zeta}.  It can be seen that $w$ being feasible to \eqref{belevel3} implies that \eqref{PnewW1} and  \eqref{PdW1}  satisfy the strong duality and have the same finite optimal value,  which is equivalent to the fact that $w\in \zeta$ when Slater condition holds for \eqref{PnewW1}.   Moreover, note that the constraints of \eqref{belevel3}  indicate that for any given $w\in \zeta$,  $\lambda$ satisfying the  constraints of \eqref{belevel3} is  optimal  to \eqref{PdW1}. Therefore the purpose of \eqref{belevel3} is to find the densest   dual optimal variable $\lambda_{6}$ for all $w\in \zeta$. Thus \eqref{belevel3} can be rewritten as
\begin{equation}\label{belevel4}
\begin{array}{lcl}
&\max\limits_{(w, \lambda)}& \mathrm{\Psi}_{\varepsilon}(\lambda_{6})\\
 &\mathrm{s.t.}& w\in \zeta, ~B^{T}\lambda_{2}-A^{T}\lambda_{3}+\lambda_{4}-\lambda_{5}=0, ~\left \|\lambda_{3}\right\|_{2}\leq \lambda_{1},\\
 & & w= \lambda_{4}+\lambda_{5}+\lambda_{6}\geq 0, ~\lambda_{i}\geq 0, ~i=1,2,4,5,6,\\
 & & \mathrm{where} ~ \lambda=(\lambda_{1},...,\lambda_{6}) ~ \mathrm{is} ~ \mathrm{optimal} ~ \mathrm{to}  \\
 & &\max_{\lambda}\lbrace-\lambda_{1}\epsilon-\lambda_{2}^{T}b+\lambda_{3}^{T}y:\left \|\lambda_{3}\right\|_{2}\leq \lambda_{1}, ~w= \lambda_{4}+\lambda_{5}+\lambda_{6},\\
 & & B^{T}\lambda_{2}-A^{T}\lambda_{3}+\lambda_{4}-\lambda_{5}=0, ~ \lambda_{i}\geq 0, ~i=1,2,4,5,6\rbrace.
\end{array}
\end{equation}
Denote the feasible set of \eqref{PdW1} by
\begin{equation}\label{dual feasible set}
\begin{split}
D(w):=&\lbrace \lambda: B^{T}\lambda_{2}-A^{T}\lambda_{3}+\lambda_{4}-\lambda_{5}=0, ~\left \|\lambda_{3}\right\|_{2}\leq \lambda_{1}, ~w= \lambda_{4}+\lambda_{5}+\lambda_{6}\geq 0,~\\
&\lambda_{i}\geq 0, ~i=1,2,4,5,6\rbrace.
\end{split}
\end{equation}
Clearly, the problem \eqref{belevel4} can be presented as
\begin{equation}\label{belevel4'}
\begin{array}{lcl}
&\max\limits_{(w, \lambda)}& \mathrm{\Psi}_{\varepsilon}(\lambda_{6})\\
 &\mathrm{s.t.}& w\in \zeta, ~\lambda\in D(w),~ \mathrm{where} ~ \lambda ~ \mathrm{is} ~ \mathrm{optimal} ~ \mathrm{to}  \\
 & &\max_{\lambda}\lbrace-\lambda_{1}\epsilon-\lambda_{2}^{T}b+\lambda_{3}^{T}y:\lambda\in D(w)\rbrace.
\end{array}
\end{equation}
   An optimal solution of \eqref{belevel4'} can be obtained by  maximizing  $\mathrm{\Psi}_{\varepsilon}(\lambda_{6})$ which is based on maximizing  $-\lambda_{1}\epsilon-\lambda_{2}^{T}b+\lambda_{3}^{T}y$ over the feasible set of \eqref{belevel4'}.  Therefore, $\mathrm{\Psi}_{\varepsilon}(\lambda_{6})$ and $-\lambda_{1}\epsilon-\lambda_{2}^{T}b+\lambda_{3}^{T}y$ are required to be maximized over the dual constraints $\lambda\in D(w)$ for all  $w\in \zeta$. To maximize both the objective functions,  we consider  the following model as the first relaxation of \eqref{belevel1}:
 \begin{equation}\label{belevel5}
\begin{array}{lcl}
&\max\limits_{(w, \lambda)}& -\lambda_{1}\epsilon-\lambda_{2}^{T}b+\lambda_{3}^{T}y+\alpha\mathrm{\Psi}_{\varepsilon}(\lambda_{6})\\
 &\mathrm{s.t.}& w\in \zeta, ~\lambda\in D(w).\\
\end{array}
\end{equation}
where $\alpha>0$ is a given small parameter.

Now we  develop the second type of relaxation of the bilevel optimization \eqref{belevel1}.
 Note that under Slater condition,  for all $w\in \zeta$, the dual objective $-\lambda_{1}\epsilon-\lambda_{2}^{T}b+\lambda_{3}^{T}y$ must be nonnegative and is homogeneous in $\lambda=(\lambda_{1},...,\lambda_{6})$. Moreover, if $w\in \zeta$, then  $-\lambda_{1}\epsilon-\lambda_{2}^{T}b+\lambda_{3}^{T}y$  has a nonnegative upper bound due to the weak duality.  Inspired by this  observation, in order to maximize both $\mathrm{\Psi}_{\varepsilon}(\lambda_{6})$ and $-\lambda_{1}\epsilon-\lambda_{2}^{T}b+\lambda_{3}^{T}y$, we may introduce  a small positive $\alpha$ and consider the following approximation:
 \begin{equation}\label{belevel7}
\begin{array}{lcl}
&\max\limits_{(w, \lambda)}& -\lambda_{1}\epsilon-\lambda_{2}^{T}b+\lambda_{3}^{T}y  \\
 &\mathrm{s.t.}& w\in \zeta, ~\lambda\in D(w),~ -\lambda_{1}\epsilon-\lambda_{2}^{T}b+\lambda_{3}^{T}y\leq \alpha \mathrm{\Psi}_{\varepsilon}(\lambda_{6}).\\
\end{array}
\end{equation}
The constraint
\begin{equation}\label{relax idea 2}
-\lambda_{1}\epsilon-\lambda_{2}^{T}b+\lambda_{3}^{T}y\leq \alpha \mathrm{\Psi}_{\varepsilon}(\lambda_{6})
\end{equation}
  implies that $\mathrm{\Psi}_{\varepsilon}(\lambda_{6})$ might be maximized when $-\lambda_{1}\epsilon-\lambda_{2}^{T}b+\lambda_{3}^{T}y$ is maximized if $\alpha$ is small and suitably chosen.

Finally, we
consider the following inequality in order to develop third type of convex relaxation.
\begin{equation}\label{relax idea 3}
-\lambda_{1}\epsilon-\lambda_{2}^{T}b+\lambda_{3}^{T}y+f(\lambda_{6})\leq \gamma,
\end{equation}
where $\gamma$ is a given positive number,  $f(\lambda_{6})$ is a certain function depending on $\varphi_{\varepsilon}((\lambda_{6})_{i})$, which satisfies the following properties:
\begin{itemize}
\setlength{\itemindent}{1em}
\item[$(I1)$] $f(\lambda_{6})$ is convex and continuous  with respect to $\lambda_{6}\in R_{+}^{n}$;
\item[$(I2)$]     maximizing  $\mathrm{\Psi}_{\varepsilon}(\lambda_{6})$ over the feasible set can be equivalently or approximately achieved by  minimizing $f(\lambda_{6})$.
\end{itemize}
There are many functions satisfying the properties $(I1)$ and $(I2)$. For instance, we may consider  the following  functions:
\begin{itemize}
\setlength{\itemindent}{1em}
\item[$(J1)$] $e^{-\mathrm{\Psi}_{\varepsilon}(\lambda_{6})}$; $(J2)$ $-\log(\mathrm{\Psi}_{\varepsilon}(\lambda_{6})+\sigma_{1})$; $(J3)$  $\frac{1}{\mathrm{\Psi}_{\varepsilon}(\lambda_{6})+\sigma_{1}}$; $(J4)$ $\frac{1}{n}\sum_{i=1}^{n} \frac{1}{\varphi_{\varepsilon}((\lambda_{6})_{i})+\sigma_{1}}$,
\end{itemize}
where $\sigma_{1}$ is a small positive number. Now we claim that the functions $(J1)$-$(J4)$ satisfy $(I1)$ and $(I2)$. Clearly, the  functions $(J1),(J2)$ and $(J3)$ satisfy  $(I2)$. Note that $$\frac{1}{\mathrm{\Psi}_{\varepsilon}(\lambda_{6})+\sigma_{1}} \leq \frac{1}{n}\sum_{i=1}^{n} \frac{1}{\varphi_{\varepsilon}((\lambda_{6})_{i})+\sigma_{1}}.$$
Thus the minimization of   $\frac{1}{n}\sum_{i=1}^{n} \frac{1}{\varphi_{\varepsilon}((\lambda_{6})_{i})+\sigma_{1}}$ is likely to imply the minimization of  $\frac{1}{\mathrm{\Psi}_{\varepsilon}(\lambda_{6})+\sigma_{1}}$, which means  the maximization of $\mathrm{\Psi}_{\varepsilon}(\lambda_{6})$. It is easy to check that the functions $(J1)$-$(J4)$ are continuous in $\lambda_{6}\geq 0$. It is also easy to check  that $(J1)$-$(J3)$ are convex for $\lambda_{6}\geq 0$. Note that for any $\varphi_{\varepsilon}((\lambda_{6})_{i})> -\sigma_{1},~i=1,...,n$, all functions $\frac{1}{\varphi_{\varepsilon}((\lambda_{6})_{i})+\sigma_{1}}$ are convex.  Therefore their sum  is convex for $\lambda_{6}\geq 0$ as well. Thus all  functions $(J1)$-$(J4)$ satisfy the two properties $(I1)$ and $(I2)$.  Moreover, the functions $(J1), (J3)$ and $(J4)$ have  finite values even  when $(\lambda_{6})_{i}\rightarrow \infty$.

 Replacing $-\lambda_{1}\epsilon-\lambda_{2}^{T}b+\lambda_{3}^{T}y\leq \alpha \mathrm{\Psi}_{\varepsilon}(\lambda_{6})$ in \eqref{belevel7} by \eqref{relax idea 3} leads to the  model
\begin{equation}\label{belevel9}
\begin{array}{lcl}
&\max\limits_{(w, \lambda)}& -\lambda_{1}\epsilon-\lambda_{2}^{T}b+\lambda_{3}^{T}y  \\
 &\mathrm{s.t.}& w\in \zeta, ~\lambda\in D(w),~-\lambda_{1}\epsilon-\lambda_{2}^{T}b+\lambda_{3}^{T}y+f(\lambda_{6})\leq \gamma.\\
\end{array}
\end{equation}
Clearly, the convexity of $f(\lambda_{6})$ guarantees that \eqref{belevel9}
is  a convex optimization. Moreover, \eqref{relax idea 3} and the  property $(I2)$ of  $f(\lambda_{6})$ imply that maximizing $-\lambda_{1}\epsilon-\lambda_{2}^{T}b+\lambda_{3}^{T}y$   is roughly equivalent to minimizing  $f(\lambda_{6})$  over the feasible set,  and thus    maximizing   $\mathrm{\Psi}_{\varepsilon}(\lambda_{6})$. The properties $(I1)$ and $(I2)$ ensure that the problem \eqref{belevel9} is computationally tractable  and is a certain relaxation of \eqref{belevel4'} and \eqref{belevel1}.

\subsection{One-step dual-density-based algorithm}
Note that the set $\zeta$ has no explicit form, and we need to deal with the set $\zeta$  to solve  three relaxation problems \eqref{belevel5}, \eqref{belevel7} and \eqref{belevel9}. First we  relax $w\in \zeta$ to $w\in R^{n}_{+}$ and obtain three convex minimization models. In this case, the difficulty for solving the problems \eqref{belevel5} and \eqref{belevel7}   is that $\mathrm{\Psi}_{\varepsilon}(\lambda_{6})$ might attain an infinite value when $w_{i}\rightarrow \infty$.  We may introduce   a  bounded merit function $\mathrm{\Psi}_{\varepsilon}\in \text{F}$  into \eqref{belevel5} and \eqref{belevel7} so that the  value of $\mathrm{\Psi}_{\varepsilon}(\lambda_{6})$ is finite.  Moreover, to avoid the infinite optimal value in the model \eqref{belevel5}, $w\in\zeta$  can be  relaxed to $-\lambda_{1}\epsilon-\lambda_{2}^{T}b+\lambda_{3}^{T}y\leq 1$ due to the weak duality. Based on the above observation, we obtain a solvable relaxation for \eqref{belevel5} and \eqref{belevel7} respectively as follows:
 \begin{equation}\label{belevell}
\begin{array}{cl}
\max\limits_{(w, \lambda)}& -\lambda_{1}\epsilon-\lambda_{2}^{T}b+\lambda_{3}^{T}y+\alpha\mathrm{\Psi}_{\varepsilon}(\lambda_{6})\\
 \mathrm{s.t.}& w\in R^{n}_{+}, ~\lambda\in D(w),~-\lambda_{1}\epsilon-\lambda_{2}^{T}b+\lambda_{3}^{T}y\leq 1.\\
\end{array}
\end{equation}
and
 \begin{equation}\label{belevel8}
\begin{array}{cl}
\max\limits_{(w, \lambda)}& -\lambda_{1}\epsilon-\lambda_{2}^{T}b+\lambda_{3}^{T}y  \\
 \mathrm{s.t.}& w\in R^{n}_{+}, ~\lambda\in D(w),~ -\lambda_{1}\epsilon-\lambda_{2}^{T}b+\lambda_{3}^{T}y\leq \alpha \mathrm{\Psi}_{\varepsilon}(\lambda_{6}).\\
\end{array}
\end{equation}
Due to the constraints \eqref{relax idea 3}, the  optimal value of the problem \eqref{belevel9} is finite if it is feasible.  By replacing $\zeta$ by $R^{n}_{+}$ in \eqref{belevel9} , we also obtain  a new  relaxation of  \eqref{belevel1}:
\begin{equation}\label{belevel10}
\begin{array}{lcl}
&\max\limits_{(w, \lambda)}& -\lambda_{1}\epsilon-\lambda_{2}^{T}b+\lambda_{3}^{T}y  \\
 &\mathrm{s.t.}& w\in R^{n}_{+},~\lambda\in D(w),~-\lambda_{1}\epsilon-\lambda_{2}^{T}b+\lambda_{3}^{T}y+f(\lambda_{6})\leq \gamma.\\
\end{array}
\end{equation}
 Thus, a new weighted $\ell_{1}$-algorithm for the model \eqref{Pnew}  is developed:

\begin{algorithm}[H]
\label{free weight dual algorithm}
\caption{One-step dual-density-based algorithm [\textbf{DDA} for short]}
 \begin{algorithmic}[4]
 \REQUIRE~~\\
 merit function $\mathrm{\Psi}_{\varepsilon}\in \textbf{F}$, matrices $A\in R^{m\times n}$ and $B\in R^{l\times n}$;\\   vectors $y\in R^{m}$ and $b\in R^{l}$,  small positive parameters $(\varepsilon,\epsilon)\in R_{++}^{2}$;

\STEP ~~\\ 1. Solve the dual-density-based problem to obtain  the vector $(w^{0},\lambda_{6}^{0})$, \\
2. Let $x^{0}\in \argmin\lbrace (w^{0})^{T}\vert x\vert: x\in T\rbrace.$
\end{algorithmic}
\end{algorithm}
\noindent In this paper, we consider the forms DDA(\MyRoman{1})-DDA(\MyRoman{3}). The corresponding  constants, the dual-density-based problems for these algorithms are listed in the following table.

\begin{table}[H]
\caption{DDA(\MyRoman{1})-DDA(\MyRoman{3})}
\label{dual table}
 \begin{center}
 \begin{tabular}{lll}
\hline\noalign{\smallskip}
Name & Constants & Dual-density-based  problem  \\
\noalign{\smallskip}\hline\noalign{\smallskip}
DDA(\MyRoman{1})             & $\alpha$          &     \eqref{belevell}     \\
DDA(\MyRoman{2})                  & $\alpha$         &      \eqref{belevel8}     \\
DDA(\MyRoman{3})                 & $\gamma$         &    \eqref{belevel10} \\
\noalign{\smallskip}\hline
\end{tabular}
\end{center} 
\end{table}

\subsection{Dual-density-based reweighted $\ell_{1}$-algorithm}
Now we develop reweighted $\ell_{1}$-algorithms for \eqref{Pnew} based on  \eqref{belevel4'}. To this need, we introduce a  bounded convex set $\mathcal{W}$ for $w$ to approximate the set $\zeta$.
By replacing $\zeta$ with $\mathcal{W}$ in the models \eqref{belevel5}, \eqref{belevel7} and \eqref{belevel9}, we  obtain the following three types of convex relaxation models of \eqref{belevel1}:
 \begin{equation}\label{belevel6}
\begin{array}{lcl}
 &\max\limits_{(w, \lambda)}& -\lambda_{1}\epsilon-\lambda_{2}^{T}b+\lambda_{3}^{T}y+\alpha\mathrm{\Psi}_{\varepsilon}(\lambda_{6})\\
 &\mathrm{s.t.}& w\in \mathcal{W}, ~\lambda\in D(w), ~-\lambda_{1}\epsilon-\lambda_{2}^{T}b+\lambda_{3}^{T}y\leq 1,
\end{array}
\end{equation}
\begin{equation}\label{belevelnew1}
\begin{array}{lcl}
 &\max\limits_{(w, \lambda)}& -\lambda_{1}\epsilon-\lambda_{2}^{T}b+\lambda_{3}^{T}y  \\
 &\mathrm{s.t.}& w\in \mathcal{W}, ~\lambda\in D(w),~ -\lambda_{1}\epsilon-\lambda_{2}^{T}b+\lambda_{3}^{T}y\leq \alpha \mathrm{\Psi}_{\varepsilon}(\lambda_{6}),
\end{array}
\end{equation}
 \begin{equation}\label{belevel11}
\begin{array}{lcl}
 &\max\limits_{(w, \lambda)}& -\lambda_{1}\epsilon-\lambda_{2}^{T}b+\lambda_{3}^{T}y  \\
 &\mathrm{s.t.}& w\in \mathcal{W}, \lambda\in D(w),~-\lambda_{1}\epsilon-\lambda_{2}^{T}b+\lambda_{3}^{T}y+f(\lambda_{6})\leq \gamma.\\
\end{array}
\end{equation}
  Inspired by \cite{zhaom2015} and \cite{zhao2016}, we can choose the following bounded convex set:
\begin{equation}\label{polytope 1}
\mathcal{W}= \biggr\lbrace w\in R_{+}^{n}: (x^{0})^{T}w\leq M, 0\leq w\leq M^{*}e\biggr\rbrace,
\end{equation}
where $x^{0}$ is the initial point,  which can be the  solution of the $\ell_{1}$-minimization \eqref{Pnew1}, and $M$, $M^{*}$ are two given  numbers such that $1\leq M\leq M^{*}$.  We also consider the set
\begin{equation}\label{polytope 2}
\mathcal{W}=\biggr\lbrace w\in R_{+}^{n}:  w_{i}\leq \frac{M}{\vert x^{0}_{i}\vert+\sigma_{2}}\biggr\rbrace,
\end{equation}
where both $M$  and $\sigma_{2}$ are two given positive numbers.  $(x^{0})^{T}w\leq M$ in \eqref{polytope 1}  and $w_{i}\leq \frac{M}{\vert x^{0}_{i}\vert+\sigma_{2}}$ in \eqref{polytope 2} are  motivated by the idea of existing reweighted algorithm  in \cite{CWB2008,zhaom2015,zhao2016}.  The  set  $\mathcal{W}$ can   be seen as not only a relaxation of $\zeta$, but also being used to ensure the boundedness of $\mathrm{\Psi}_{\varepsilon}(\lambda_{6})$.
Based on \eqref{polytope 1} and \eqref{polytope 2}, we  update $\mathcal{W}$   in the algorithms either as:
\begin{equation}\label{polytope 1l}
\mathcal{W}^{k}= \biggr\lbrace w\in R_{+}^{n}: (x^{k-1})^{T}w\leq M,~ 0\leq w\leq M^{*}e\biggr\rbrace,
\end{equation}
or
\begin{equation}\label{polytope 2l}
\mathcal{W}^{k}=\biggr\lbrace w\in R_{+}^{n}:  w_{i}\leq \frac{M}{\vert x^{k-1}_{i}\vert+\sigma_{2}}\biggr\rbrace.
\end{equation}
This yields the following  algorithm (DRA for short).

\begin{algorithm}[H]
\caption{ Dual-density-based reweighted  $\ell_{1}$-algorithm [\textbf{DRA}]}
 \begin{algorithmic}[4]
 \REQUIRE~~\\
merit function $\mathrm{\Psi}_{\varepsilon}\in \textbf{F}$,    matrices  $A\in R^{m\times n}$ and $B\in R^{l\times n}$;\\
  vectors $y\in R^{m}$ and $b\in R^{l}$,  small  positive parameters $(\varepsilon,\epsilon)\in R^{2}_{++}$;\\
      the iteration index $k$, the largest number of iteration $k_{\max}$;
 \INITIAL ~~  \\
  1. Solve the dual-density-based problem  to get  $w^{0}$;\\
  2. Solve the weighted $\ell_{1}$-minimization $ \min\lbrace (w^{0})^{T}\vert x\vert: x\in T\rbrace$ to get   $x^{0}$ and $\mathcal{W}^{1}$.\\

\ENSURE ~~\\

At the current iterate $x^{k-1}$, \\
1. Solve the dual-density-based weighted problem with $\mathcal{W}^{k}$ to obtain $(w^{k},\lambda_{6}^{k})$,\\

2. Solve the $\ell_{1}$-minimization $\min\lbrace (w^{k})^{T}\vert x\vert: x\in T\rbrace$ to get $x^{k}$;\\

3. Update $\mathcal{W}^{k+1}$ and repeat the above iteration until $k=k_{\max}$ (or certain other stopping criterion is met).
\end{algorithmic}
\end{algorithm}

The initial step of DRA is to solve DDA and to get the initial weight $w^{0}$ and the set $\mathcal{W}^{1}$. Different choice of the dual-density-based problems, dual-density-based weighted problem  and  the set $\mathcal{W}$ yields different forms of DRA. In this paper, we consider the following forms of   DRA(\MyRoman{1})-DRA(\MyRoman{6}). The corresponding constants,  $\mathcal{W}$, DDA and the dual-density-based weighted problems for these algorithms are listed in the following table.
\begin{table}[H]
\caption{DRA(\MyRoman{1})-DRA(\MyRoman{6})}
\label{dual table 1}
\begin{center}
\begin{tabular}{lllcc}
\hline\noalign{\smallskip}
Name & Constants & DDA & $\mathcal{W}$ & Dual-density-based weighted problem  \\
\noalign{\smallskip}\hline\noalign{\smallskip}
DRA(\MyRoman{1})             & $\alpha,M, M^{*}$          & DDA(\MyRoman{1})                & \eqref{polytope 1l} & \eqref{belevel6}           \\
DRA(\MyRoman{2})                  & $\alpha,\sigma_{2},M$         &    DDA(\MyRoman{1})                            & \eqref{polytope 2l} & \eqref{belevel6}         \\
DRA(\MyRoman{3})                 & $\alpha,M, M^{*}$         & DDA(\MyRoman{2})                              & \eqref{polytope 1l} & \eqref{belevelnew1}        \\
DRA(\MyRoman{4})                & $\alpha,\sigma_{2},M$         & DDA(\MyRoman{2})                              & \eqref{polytope 2l} & \eqref{belevelnew1}    \\
DRA(\MyRoman{5})                & $\gamma, M, M^{*}$         & DDA(\MyRoman{3})                              & \eqref{polytope 1l} & \eqref{belevel11}    \\
DRA(\MyRoman{6})                & $\gamma,\sigma_{2},M$         & DDA(\MyRoman{3})                              & \eqref{polytope 2l} & \eqref{belevel11} \\
\noalign{\smallskip}\hline
\end{tabular}
\end{center}
\end{table}

\noindent Notice  that $w$ is restricted in the bounded set $\mathcal{W}$ so that the optimal value of \eqref{belevel6} cannot be infinite. Therefore,  we can use the bounded or unbounded  merit functions in $\mathrm{\Psi}_{\varepsilon}\in \textbf{F}$, for example, \eqref{merit function eq4}, \eqref{merit function eq1}, \eqref{merit function eq2} and \eqref{merit function eq3}. In addition, $M$ can not be too small. 	If $M$ is a sufficiently small positive number, there might be  a gap between the maximum of  $-\lambda_{1}\epsilon-\lambda_{2}^{T}b+\lambda_{3}^{T}y$ and the maximum of $\mathrm{\Psi}_{\varepsilon}(\lambda_{6})$ over the feasible set.

  The existing reweighted $\ell_{1}$-algorithm, RA, always needs an initial iterate, which is often  obtained by solving a simple  $\ell_{1}$-minimization. Unlike these existing methods, DRA(\MyRoman{1})-DRA(\MyRoman{6}) can create an initial iterate by themselves. 

\section{Numerical experiments}\label{section numerical}
 In  this section,  by choosing proper parameters and merit functions, the performance of the dual-density-based reweighted $\ell_{1}$-algorithms DRA(\MyRoman{1})-DRA(\MyRoman{6}) will be demonstrated. We use the random examples of convex sets $T$ in our experiments. We first set the noise level $\epsilon$ and the parameter $\varepsilon$ of merit functions. The sparse vector $x^{*}$ and the entries of $A$ and $B$ (if $B$ is not deterministic)  are generated from  Gaussian  random variables with zero mean and unit variance. For  each generated  $(x^{*}, A, B)$, we set $y$ and $b$ as follows:
\begin{equation}\label{xABexample}
y=Ax^{*}+\frac{c_{1}\epsilon}{\Vert c\Vert_{2}}c, ~Bx^{*}+d=b,
\end{equation}
 where   $d\in R^{l}_{+}$ is generated  as  absolute  Gaussian  random variables with zero  mean and  unit variance, and $c_{1}\in R$ and $c\in R^{m}$ are generated as  Gaussian random variables with zero mean and  unit variance. Then the convex set $T$ is generated, and all examples of $T$ are generated  this way.    We  use
\begin{equation}\label{criterion1'}
\left\|x'-x^{*}\right\| / \Vert x^{*}\Vert\leq 10^{-5}
\end{equation}
 as our default stopping criterion where $x'$ is the solution found by the  algorithm, and  one success is counted as long as  \eqref{criterion1'} is satisfied. In our experiments, we make 200 random examples for each sparsity level. All  the algorithms are  implemented  in Matlab 2018a, and all the convex problems  are solved by CVX   (Grant and Boyd \cite{cvx}).

To  demonstrate the performance of the dual-density-based reweighted $\ell_{1}$-\-algo\-rithms listed in  Table \ref{dual table 1}, we mainly consider the two cases in our  experiments
\begin{itemize}
\setlength{\itemindent}{1em}
\item[(N1)] $A\in R^{50\times 200}$, $B=0$ and $b=0$;
\item[(N2)]    $A\in R^{50\times 200}$, $B\in R^{50\times 200}$.
\end{itemize}
For all cases,  we implement the algorithms DRA(\MyRoman{1})-DRA(\MyRoman{6}), and compare their performance in finding the sparse vectors in $T$ with  $\ell_{1}$-minimization and the algorithm  RA with different merit functions.

\subsection{Merit functions and parameters}\label{section default parameters}
The default parameters and merit functions in DRA(\MyRoman{1}) and DRA(\MyRoman{2}) are set as  that of the  algorithms in \cite{zhao2016}. We set \eqref{merit function eq1} as the default merit function for   DRA(\MyRoman{3}) and DRA(\MyRoman{4}), and set $(J3)$ with
\begin{equation}\label{ND2fchoice}
f(\lambda_{6})=\frac{1}{\mathrm{\Psi}_{\varepsilon}(\lambda_{6})+\sigma_{1}}, ~ \mathrm{\Psi}_{\varepsilon}(\lambda_{6})=\sum_{i=1}^{n}\frac{(\lambda_{6})_{i}}{(\lambda_{6})_{i}+\varepsilon}, ~\lambda_{6}\in R_{+}^{n}
\end{equation}
 as the default  function   for  DRA(\MyRoman{5}) and DRA(\MyRoman{6}). We choose  the noise level $\epsilon=10^{-4}$ for both cases.
The default parameters for each dual-density-based reweighted $\ell_{1}$-algorithm are summarized in the following table:
\begin{table}[H]
\caption{Default parameters in algorithms}
\label{parameter table}
\begin{center}
\begin{tabular}{clclllll}
\hline\noalign{\smallskip}
Algorithm/Parameter & $\alpha$ & $\gamma$ & $M$ & $M^{*}$  & $\sigma_{1}$ & $\sigma_{2}$ & $\varepsilon$  \\
\noalign{\smallskip}\hline\noalign{\smallskip}
  DRA(\MyRoman{1}) & $10^{-8}$ & & $10^{2}$ & $10^{3}$ & & & $10^{-15}$      \\
 DRA(\MyRoman{2}) & $10^{-8}$ & & $10^{2}$&  & &$10^{-1}$  & $10^{-15}$     \\
  DRA(\MyRoman{3}) & $10^{-5}$ & & $10$& $10$ & &&  $10^{-15}$     \\
   DRA(\MyRoman{4}) & $10^{-5}$ & & $10$& &&$10^{-1}$  &  $10^{-15}$    \\
    DRA(\MyRoman{5}) &  & $1$& $10$ & $10$&   $10^{-1}$&& $10^{-15}$    \\
     DRA(\MyRoman{6}) &  & $1$& $10$ &  & $10^{-1}$&  $10^{-1}$& $10^{-15}$ \\
\noalign{\smallskip}\hline
\end{tabular}
\end{center}
\end{table}
\noindent The algorithms in the following table will be compared with  DRA(\MyRoman{1})-DRA(\MyRoman{6}).
 
\begin{table}[H]
\caption{Algorithms to be compared}
\label{reweighted table}
\begin{center}
\begin{tabular}{ccc}
\hline\noalign{\smallskip}
Name & Merit Function & (Reweighted) Methods  \\
\noalign{\smallskip}\hline\noalign{\smallskip}
 $\ell_{1}$                 & $\left\|x\right\|_{1} $        & $\ell_{1}$-minimization       \\
CWB              &     $\sum_{i=1}^{n} \log(\vert x_{i}\vert+\varepsilon)$      &         RA \\
ARCTAN                 & \eqref{merit function eq3}         &        RA \\
\noalign{\smallskip}\hline
\end{tabular}
\end{center}
\end{table}

Cand\`{e}s,  Wakin and Boyd  in \cite{CWB2008} developed a reweighted algorithm which is referred to as CWB in this section. From the perspective of the reweighted $\ell_{1}$-algorithm (RA) in \cite{zhaoli2012}, CWB is a special case of  RA using the merit function $\sum_{i=1}^{n}$ $\log(\vert x_{i}\vert+\varepsilon)$. The ARCTAN is also a special case of RA using the function   \eqref{merit function eq3} as the merit function for sparsity.  CWB, ARCTAN and $\ell_{1}$-minimization \eqref{Pnew1} will be compared with DRA(\MyRoman{1})-DRA(\MyRoman{6}) in  sparse vector recovery in this section. The parameter $\varepsilon$ in  RA  is set to $10^{-1}$ or $10^{-5}$, and the remaining parameters  are the same as DRA.

\subsection{Case $\mathrm{(N1)}$: }\label{case1}
\begin{figure}[H]
\centering
\begin{minipage}[t]{0.4\textwidth}
\centering
\includegraphics[width=4.5cm]{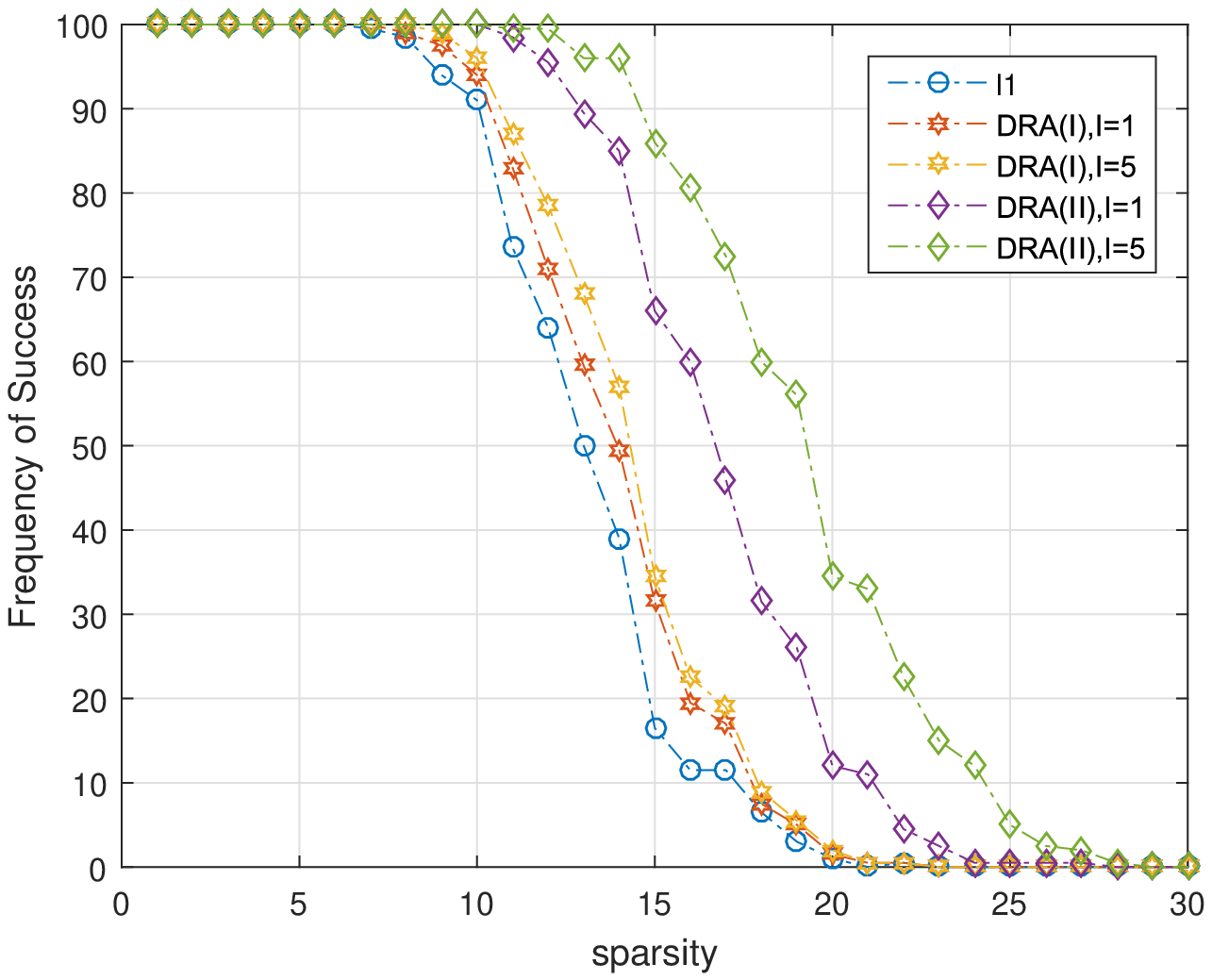}\caption*{(i) DRA(\MyRoman{1}) and DRA(\MyRoman{2})}
\end{minipage}
\begin{minipage}[t]{0.4\textwidth}
\centering
\includegraphics[width=4.5cm]{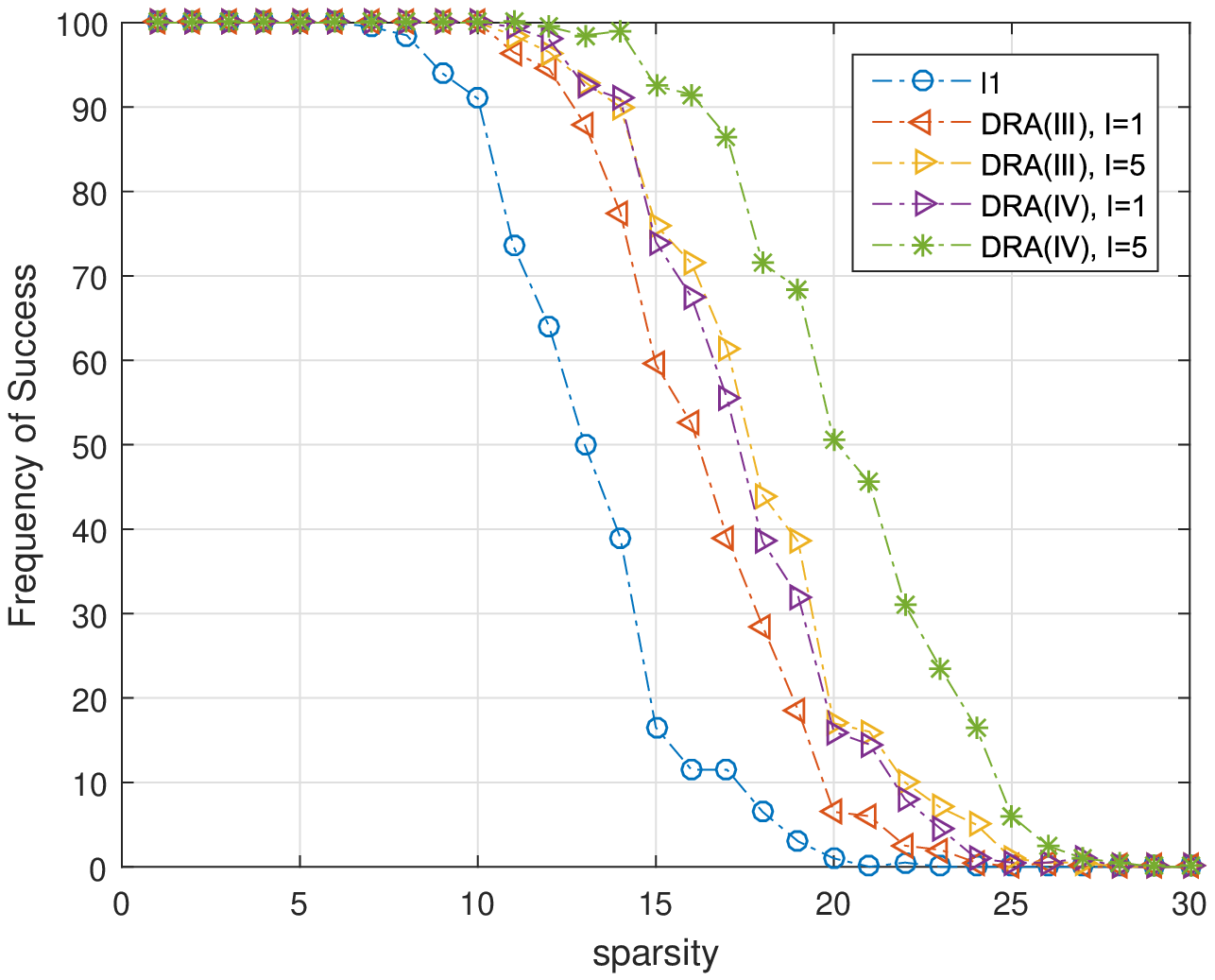}
\caption*{(ii) DRA(\MyRoman{3}) and DRA(\MyRoman{4})}
\end{minipage}
\begin{minipage}[t]{0.4\textwidth}
\centering
\includegraphics[width=4.5cm]{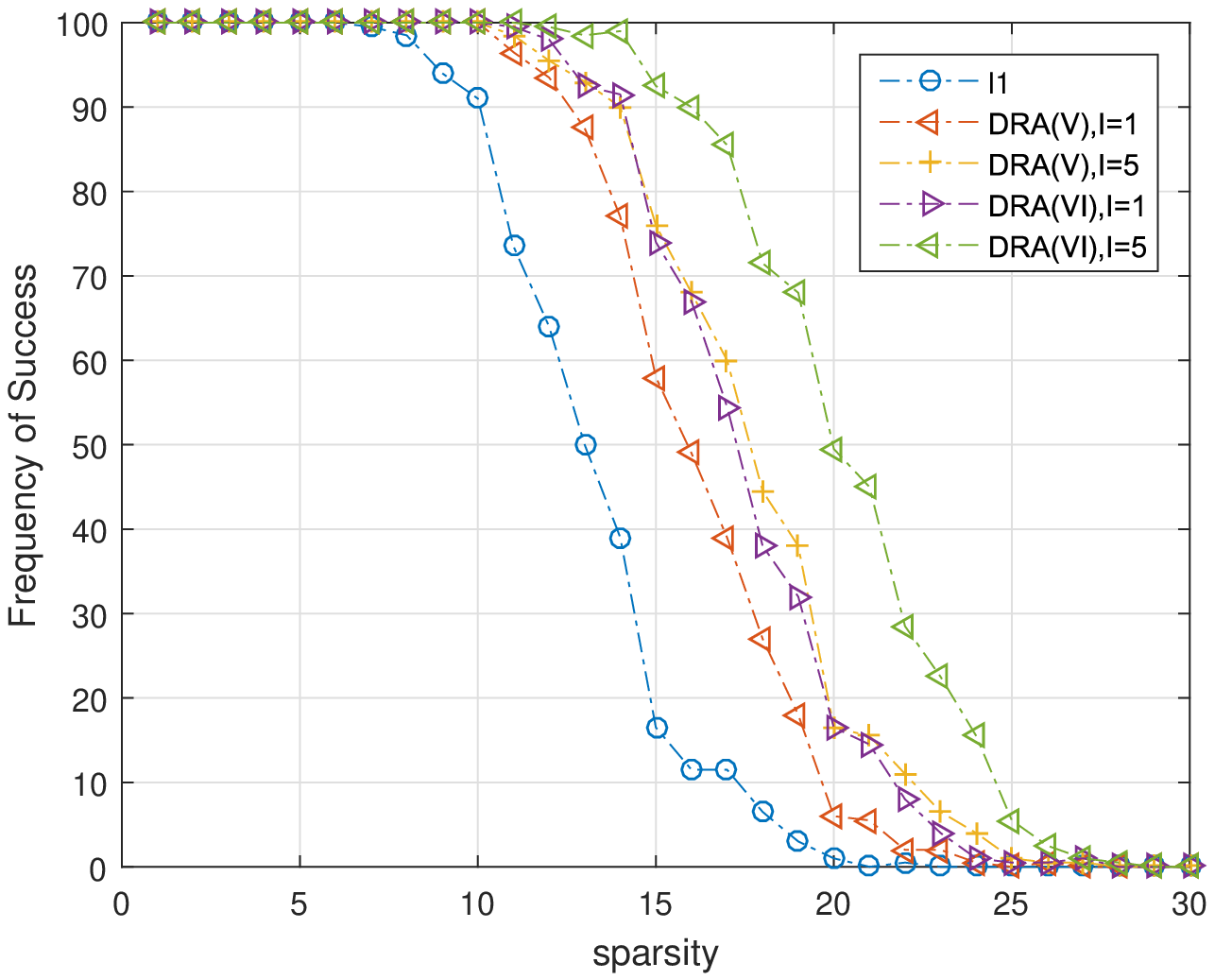}
\caption*{(iii) DRA(\MyRoman{5}) and DRA(\MyRoman{6})}
\end{minipage}
\begin{minipage}[t]{0.4\textwidth}
\centering
\includegraphics[width=4.5cm]{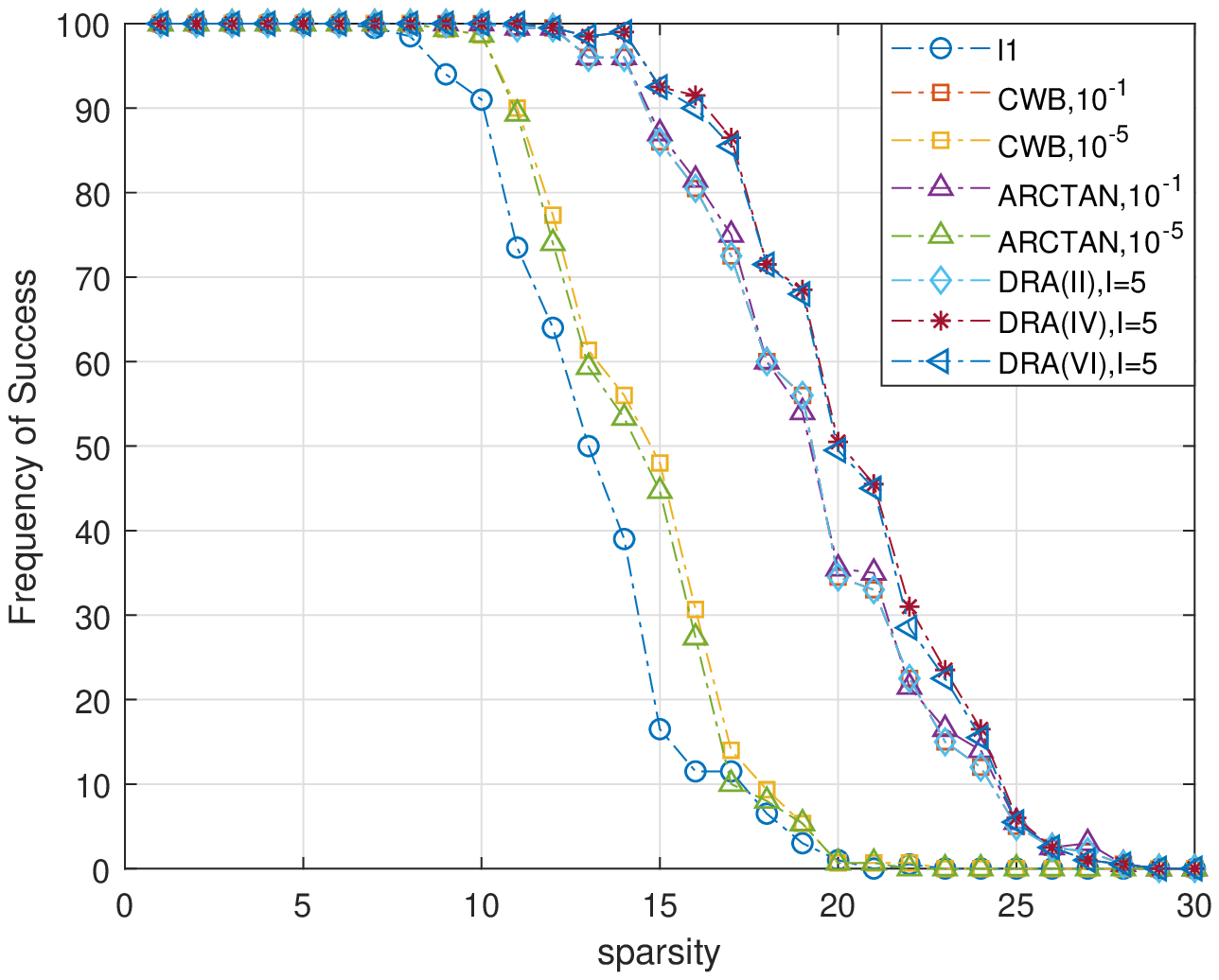}\caption*{(iv) CWB, ARCTAN }
\end{minipage}
\caption{(i)-(iii) Comparison of the performance of the dual-density-based reweighted $\ell_{1}$-algorithms by performing 1 iteration and 5 iterations respectively. (iv) Comparison of DRA and RA. }
\label{fig:15}
\end{figure}
Now we perform numerical experiments to show the behaviors of  the dual-density-based reweighted $\ell_{1}$-algorithms in two cases (N1) and (N2).
Note that in the case of (N1), the model \eqref{Pnew}  is reduced to the  sparse model  (C2). The numerical results  are given in Figure \ref{fig:15} (i)-(iii),
 Note that there are five legends in each figure (i)-(iii),  corresponding to $\ell_{1}$-minimization,  the dual-density-based reweighted $\ell_{1}$-algorithms with  one iteration or  five iterations. For instance, in (ii), we compare  DRA(\MyRoman{3}) and DRA(\MyRoman{4}) which all perform either one iteration or five iterations. For example, (DRA(\MyRoman{3}),1) and  (DRA(\MyRoman{3}),5) represent  DRA(\MyRoman{3}) with one iteration and five iterations, respectively.
 
It can be seen that the dual-density-based reweighted $\ell_{1}$-algorithms are performing better when the number of iteration is increased and all of them outperform  $\ell_{1}$-minimization in our experiment environment, while the performance of  DRA(\MyRoman{1}) with one or five iterations   is  similar  to  the performance of $\ell_{1}$-minimization. (i)-(iii) indicate the same phenomena: the  algorithms based on \eqref{polytope 2l} might achieve more  improvement  than the ones based on \eqref{polytope 1l} when the number of iteration is increased.
   For example, in (iii), the success rate of DRA(\MyRoman{6}) with five iterations has improved by nearly $25\%$ compared with  those with one iteration for each sparsity from $14$ to $20$, while DRA(\MyRoman{5}) has only  improved its performance by $10\%$ after increasing the number of iterations.
We filter the algorithms with the best performance from (i)-(iii) in Figure \ref{fig:15} and merge them into  (iv) together with CWB and ARCTAN in Figure \ref{fig:15}. It can be seen that    DRA(\MyRoman{4}) and DRA(\MyRoman{6}) outperform CWB  and ARCTAN, especially as $\varepsilon$ in CWB and ARCTAN is relatively small, and they also outperform the $\ell_{1}$-minimization as well.

\subsection{Case $(\mathrm{N2})$: }\label{case5}
\begin{figure}[H]
\centering
\begin{minipage}[t]{0.4\textwidth}
\centering
\includegraphics[width=4.5cm]{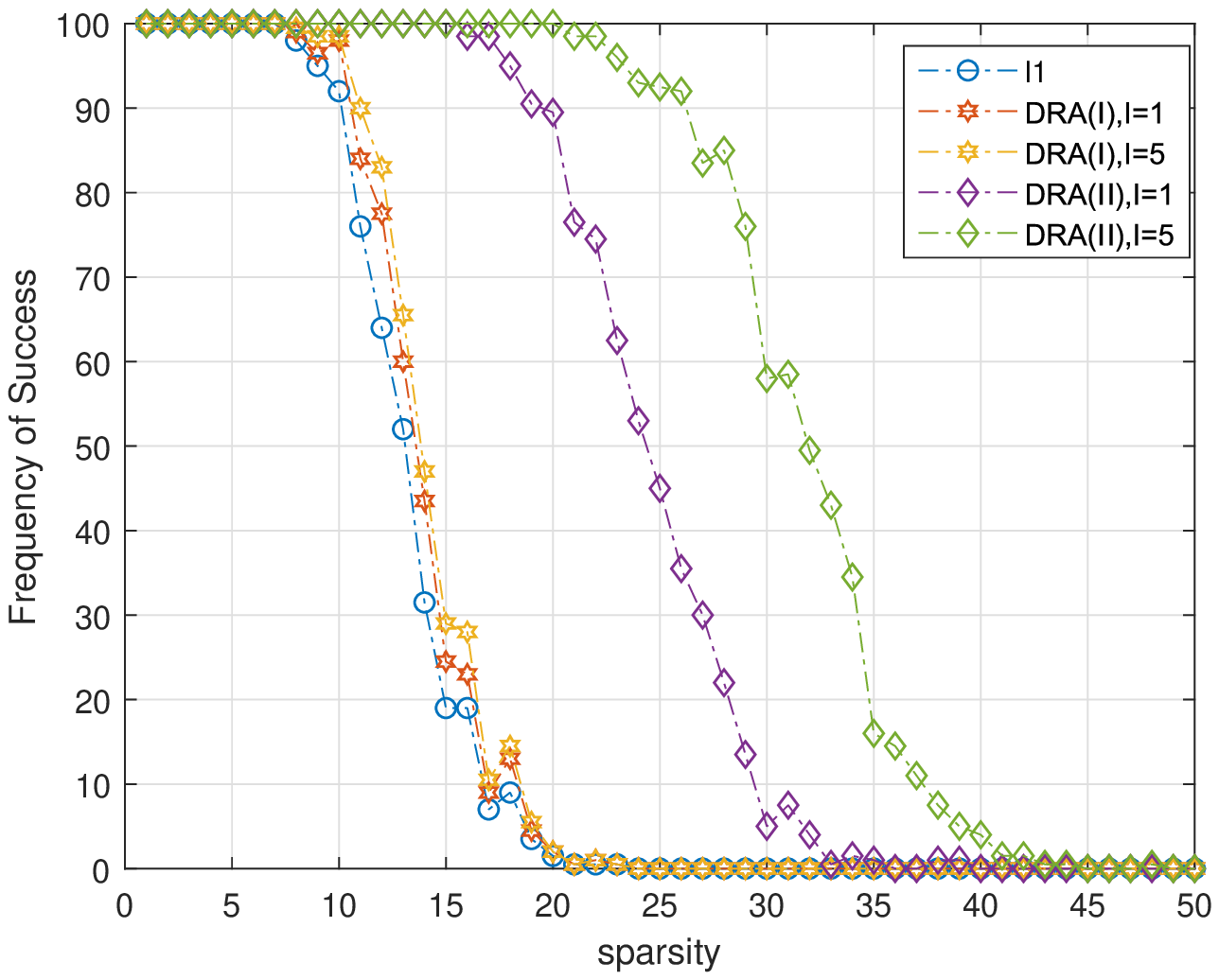}\caption*{(i) DRA(\MyRoman{1}) and DRA(\MyRoman{2})}
\end{minipage}
\begin{minipage}[t]{0.4\textwidth}
\centering
\includegraphics[width=4.5cm]{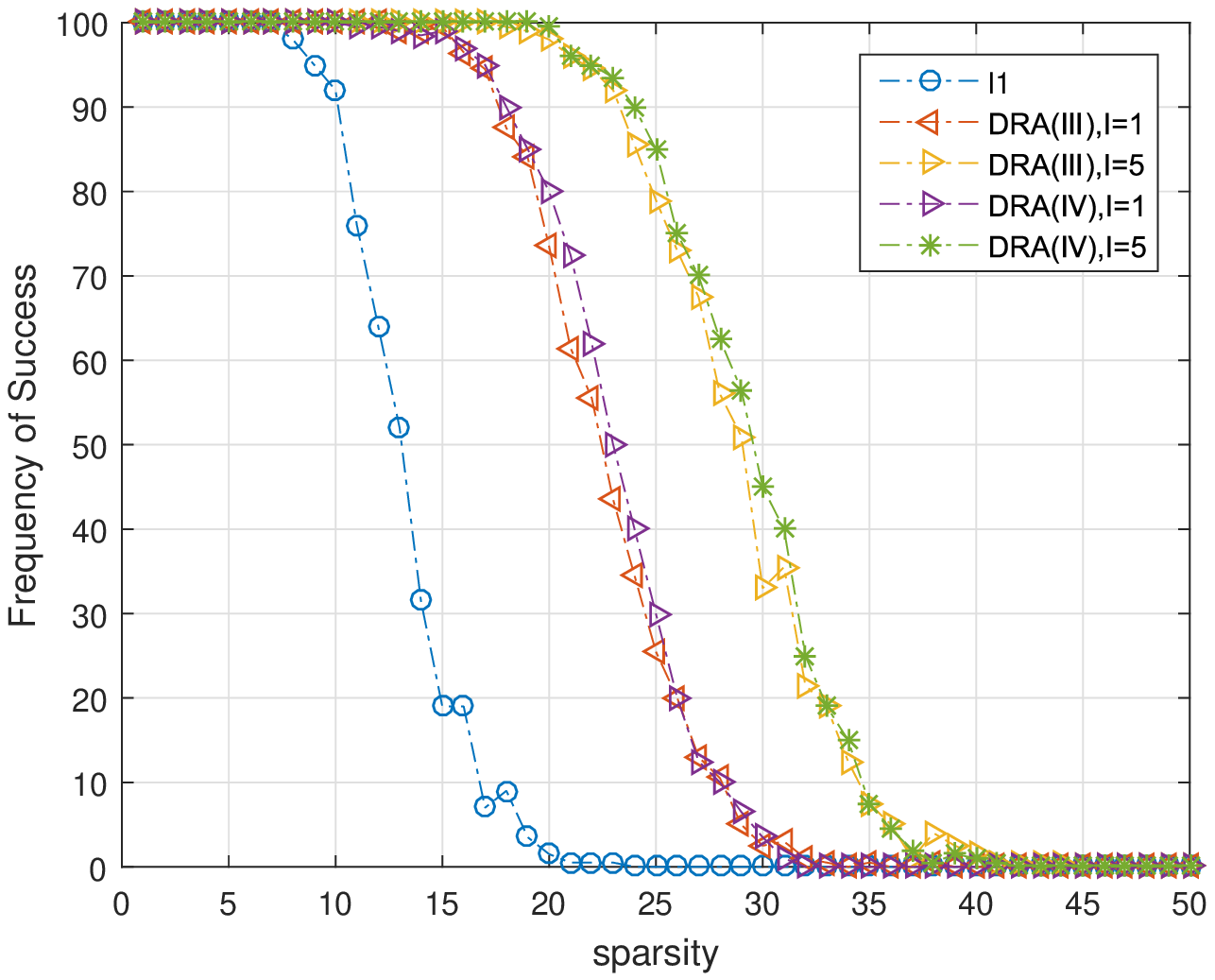}
\caption*{(ii) DRA(\MyRoman{3}) and DRA(\MyRoman{4})}
\end{minipage}
\begin{minipage}[t]{0.4\textwidth}
\centering
\includegraphics[width=4.5cm]{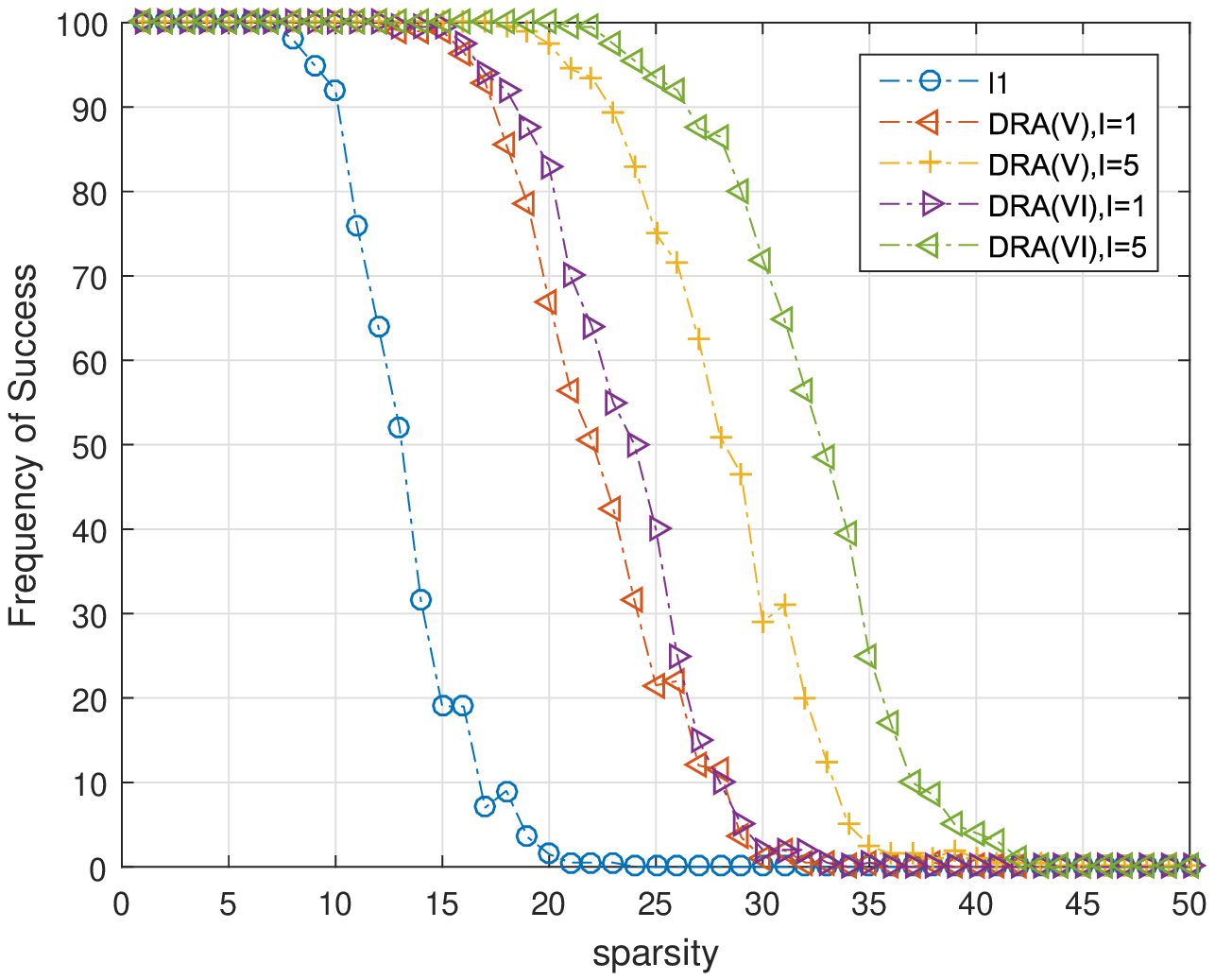}
\caption*{(iii) DRA(\MyRoman{5}) and DRA(\MyRoman{6})}
\end{minipage}
\begin{minipage}[t]{0.4\textwidth}
\centering
\includegraphics[width=4.5cm]{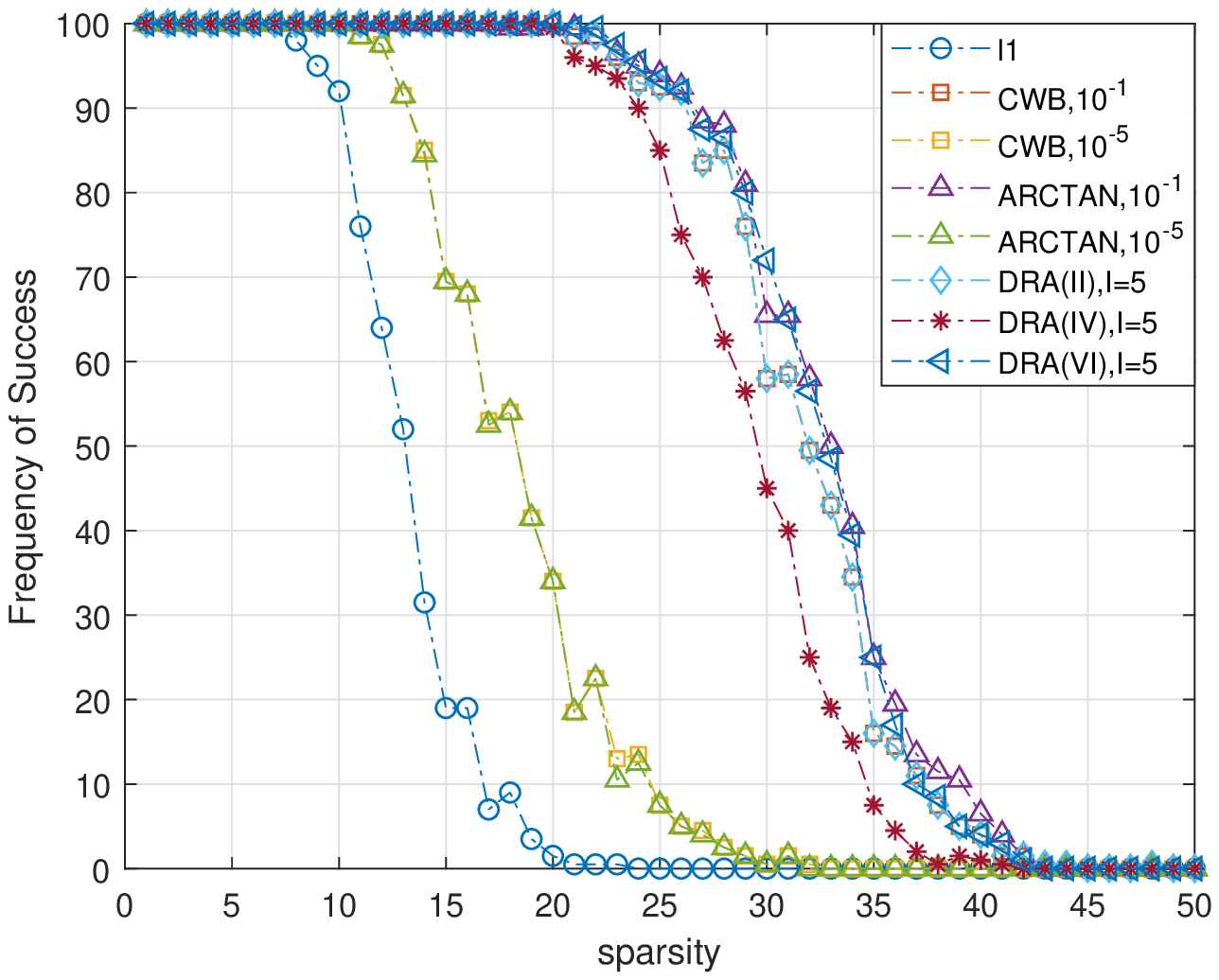}\caption*{(iv) CWB, ARCTAN }
\end{minipage}
\caption{(i)-(iii) Comparison of the performance of DRA  with one iteration and five iterations. (iv) Comparison of the performance of the DRA  and RA.}
\label{fig:23}
\end{figure}

Although the performance of ARCTAN and  DRA(\MyRoman{6}) is slightly better than that of DRA(\MyRoman{2}) and CWB in the case (N2), these  algorithms can compete to  each other in finding sparse vectors at high sparsity level in many situations. The other behaviors are similar to the case (N1).  We  compare the reweighted $\ell_{1}$-algorithms with updating rule \eqref{polytope 1l} and \eqref{polytope 2l}, which are shown in (i) and  (ii)  in Figure \ref{fig:24}, respectively. For the  algorithms using  \eqref{polytope 1l}, when executing 5 iterations, Figure \ref{fig:24} (i) shows that DRA(\MyRoman{3}) and DRA(\MyRoman{5}) perform much better than DRA(\MyRoman{1}).  For the  algorithms using  \eqref{polytope 2l}, when executing 5 iterations, Figure \ref{fig:24} (ii) indicates that the success rates of finding the sparse vectors in $T$ by DRA(\MyRoman{2}) and DRA(\MyRoman{6}) are very similar.
\begin{figure}[H]
\centering
\begin{minipage}[t]{0.4\textwidth}
\centering
\includegraphics[width=5cm]{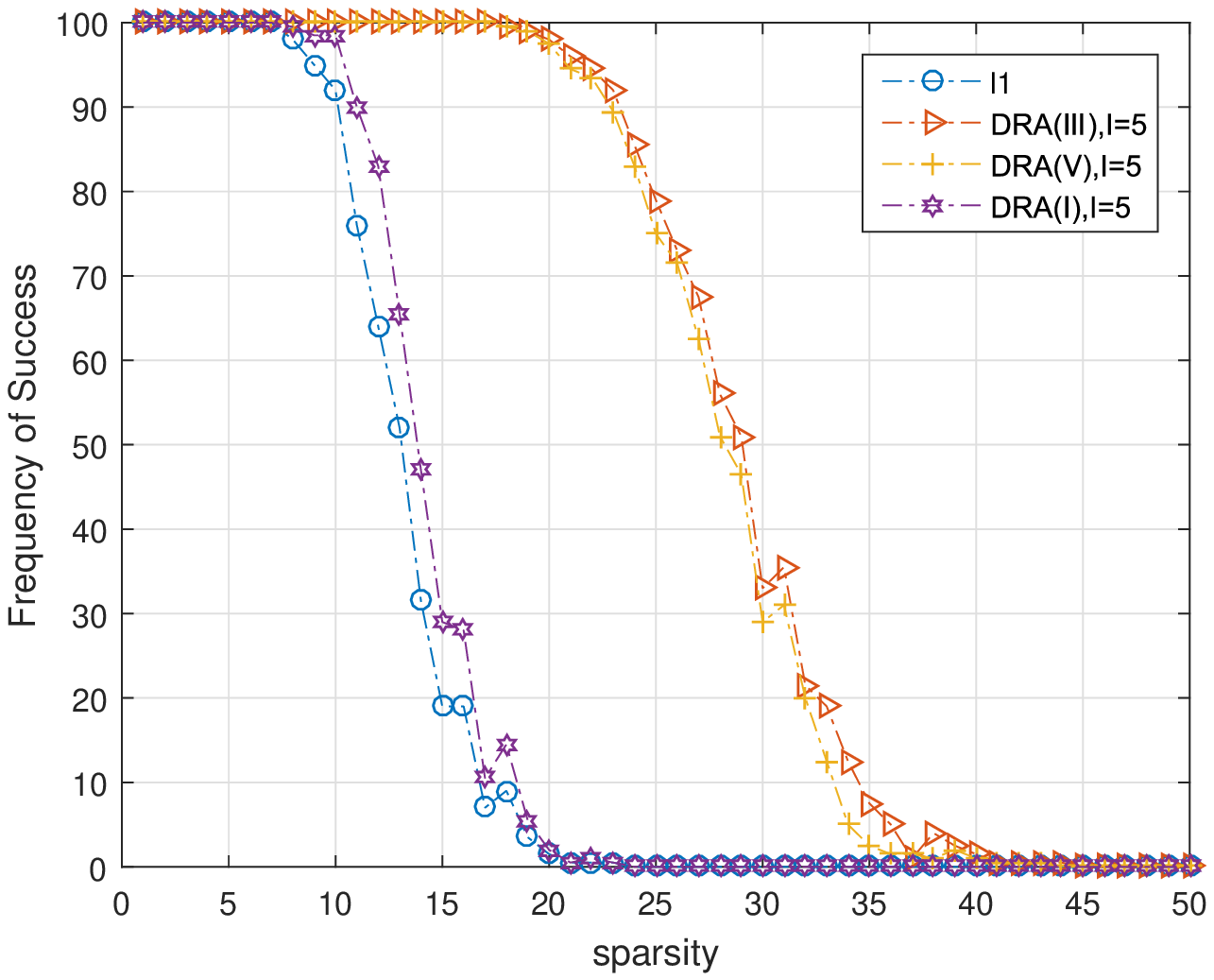}\caption*{ (i)  Algorithms with rule \eqref{polytope 1l}}
\end{minipage}
\begin{minipage}[t]{0.4\textwidth}
\centering
\includegraphics[width=5cm]{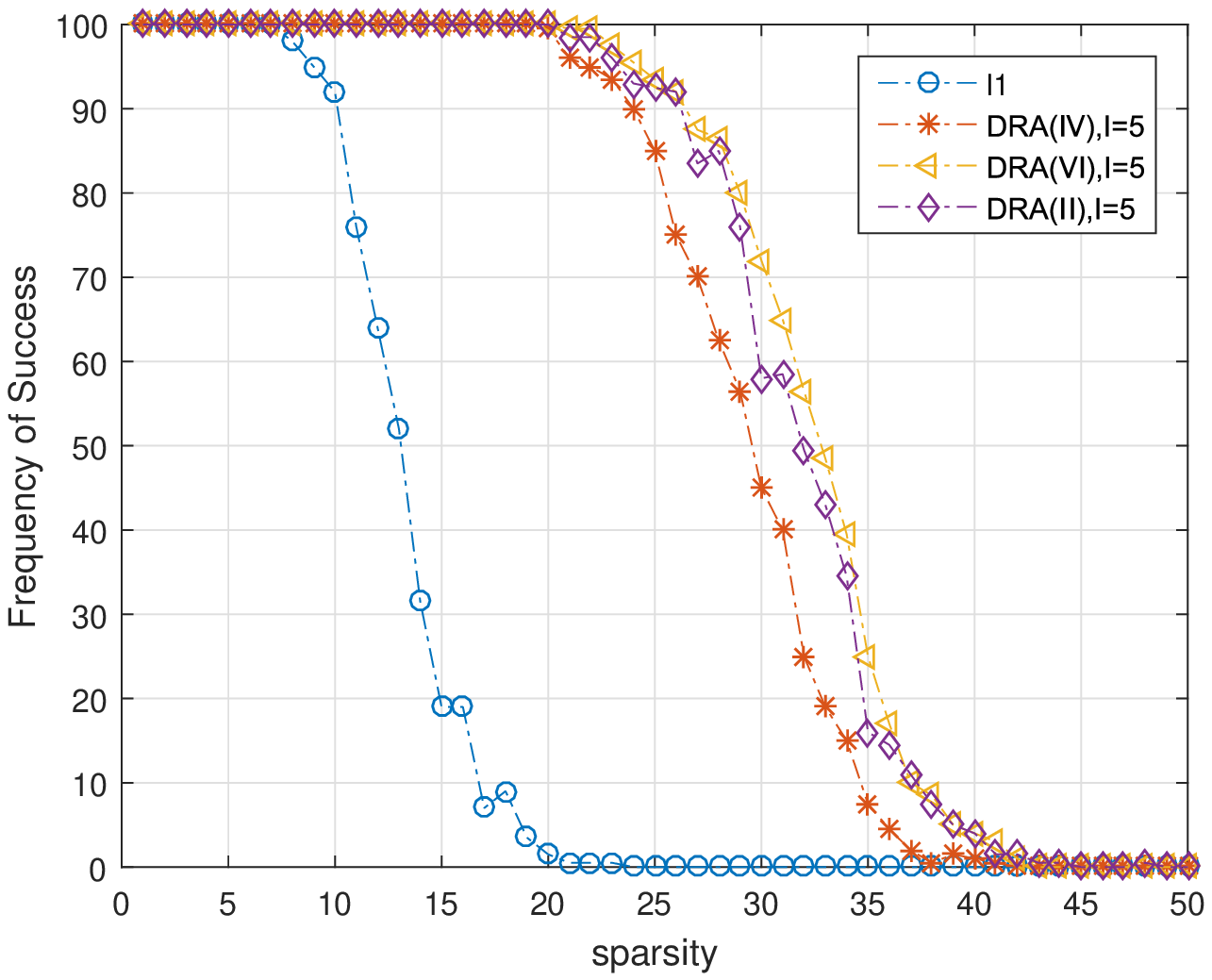}
\caption*{(ii) Algorithms with  rule \eqref{polytope 2l}}
\end{minipage}
\caption{ Comparison of the performance of DRA with \eqref{polytope 1l} or \eqref{polytope 2l}}
\label{fig:24}
\end{figure}
Finally, we carry out experiment to show  how the parameter $\varepsilon$ of merit functions affect the performance of locating the sparse vectors in $T$ by dual-density-based reweighted $\ell_{1}$-algorithms. In Figure \ref{fig:25}, some  numerical results for dual-density-based reweighted algorithms with different $\varepsilon$  indicate that the performance of the
DRA-typed algorithm is relatively insensitive to the choice of  small $\varepsilon$. Experiments reveals that when $\varepsilon \leq 10^{-10} $, the performance of CWB and ARCTAN are almost identical to that of  $\ell_{1}$-minimization, which is also observed in (iv) in Figures \ref{fig:15}  when $\varepsilon=10^{-5}$.

\begin{figure}[H]
\centering
\begin{minipage}[t]{0.4\textwidth}
\centering
\includegraphics[width=5cm]{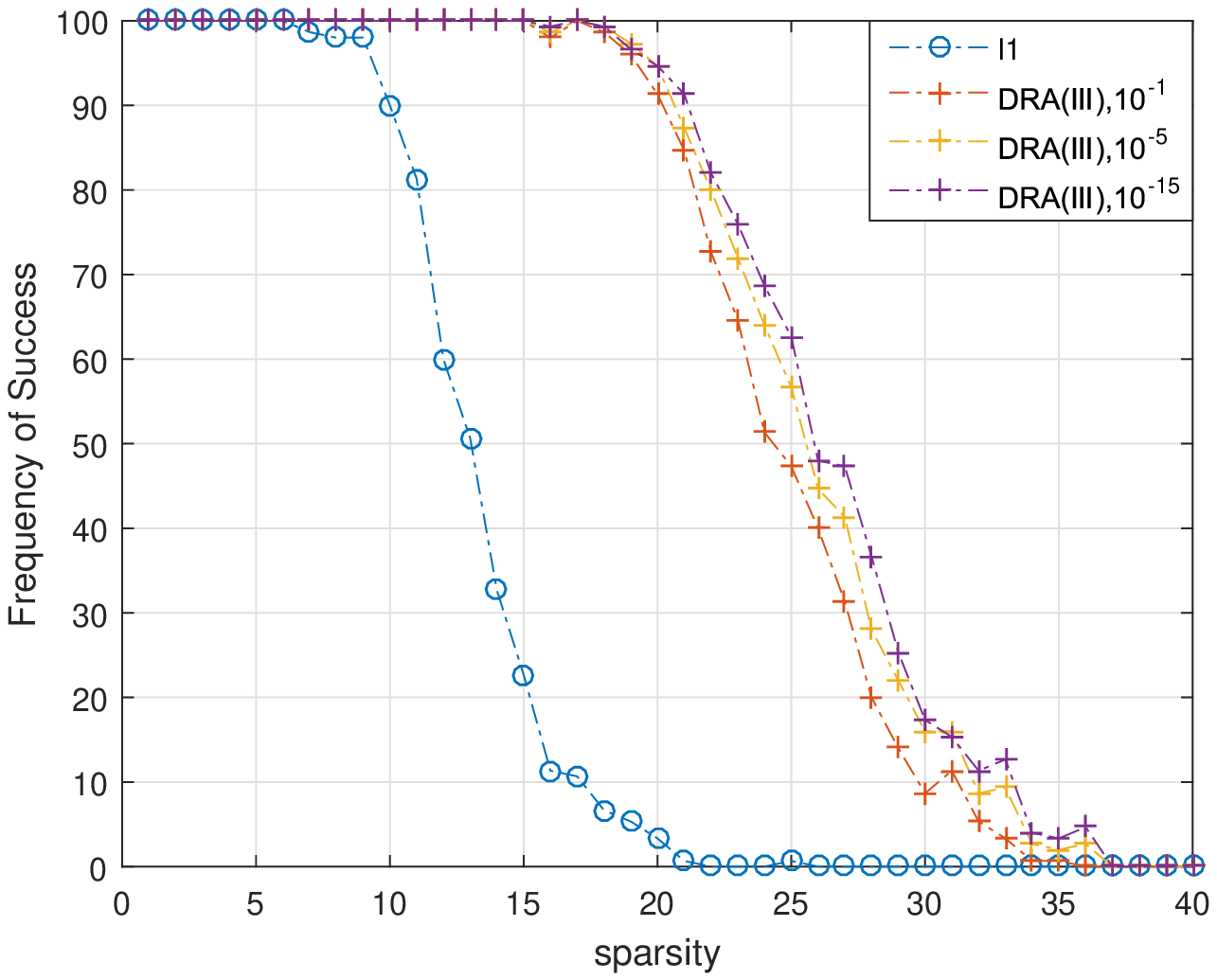}\caption*{ (i)  DRA(\MyRoman{3})}
\end{minipage}
\begin{minipage}[t]{0.4\textwidth}
\centering
\includegraphics[width=5cm]{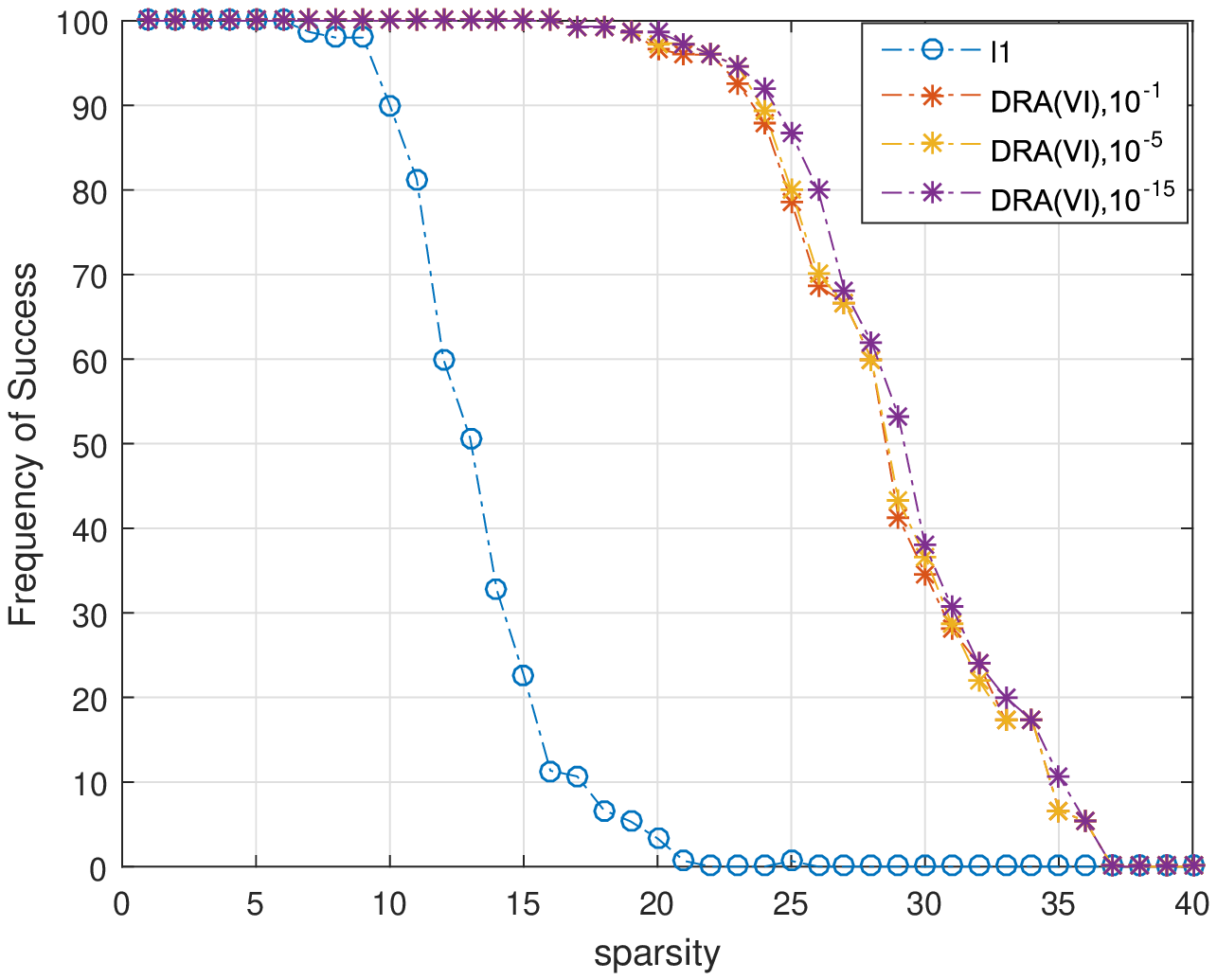}
\caption*{(ii) DRA(\MyRoman{6})}
\end{minipage}
\caption{ Comparison of the performance of DRA  with different  $\varepsilon$}
\label{fig:25}
\end{figure}

\section{Conclusions}

In this paper, we have studied  a class of algorithms for the   $\ell_{0}$-minimization problem \eqref{Pnew}. The one-step dual-density-based algorithms (DDA) and the dual-density-based reweighted $\ell_{1}$-algorithms (DRA) are developed. These algorithms are developed based on the new relaxation of the equivalent bilevel optimization of the underlying $\ell_0$-minimization problem. Unlike  RA,  the DRA can automatically generate an initial iterate instead of obtaining the initial iterate by solving $\ell_{1}$-minimization. Numerical experiments show that in some cases such as (N1) and (N2), the dual-density-based methods  proposed in this paper can perform  better than   $\ell_{1}$-minimization in solving the sparse optimization problem (\ref{Pnew}),  and can be  comparable to some existing reweighted $\ell_1$-methods. Although the experiments have shown that DRA-typed algorithms outperform $\ell_{1}$-minimization and some classic reweighted $\ell_{1}$-algorithms, there still exist some  future work to do. For example, the convergence and the stability of DRA-typed algorithms are worthwhile future work, which  might be investigated under certain assumptions such as the so-called restricted weak range space property (see, e.g., \cite{xuzhao2020}).

\end{document}